\documentclass[12pt,reqno,a4paper]{article}
\usepackage{mathtools}
\usepackage{enumerate}
\usepackage{graphicx} 
\usepackage{amsmath,amssymb,amsfonts,amsthm}
\usepackage{xcolor}
\usepackage[
inner=2.3cm,
outer=2.3cm,
top=2.7cm,
marginparwidth=2.5cm,
marginparsep=0.1cm
]
{geometry}
\allowdisplaybreaks
\usepackage{mathrsfs}

\usepackage{cite}

\newcommand{\ve}{\varepsilon}
\newcommand{\round}[1]{\lfloor#1\rfloor}
\newcommand{\Lip}{{\mathrm{Lip}}}

\newcommand\bN{{\mathbb N}}

\newcommand\bP{{\mathbb P}}

\newcommand\bR{{\mathbb R}}

\newcommand\bZ{{\mathbb Z}}

\newcommand\cA{{\mathcal A}}
\newcommand\cB{{\mathcal B}}
\newcommand\cC{{\mathcal C}}
\newcommand\cD{{\mathcal D}}
\newcommand\cE{{\mathcal E}}
\newcommand\cF{{\mathcal F}}
\newcommand\cG{{\mathcal G}}
\newcommand\cH{{\mathcal H}}
\newcommand\cI{{\mathcal I}}
\newcommand\cJ{{\mathcal J}}

\newcommand\cL{{\mathcal L}}

\newcommand\cP{{\mathcal P}}
\newcommand\cQ{{\mathcal Q}}

\newcommand\cT{{\mathcal T}}

\newcommand\cW{{\mathcal W}}

\newcommand\fD{{\mathfrak D}}

\newcommand\fF{{\mathfrak F}}

\newtheoremstyle{special}%
{}%
{}%
{}%
{}%
{\scshape}%
{.}%
{.5em}%
{}
\newtheorem{theorem}{Theorem}

\newtheorem{lemma}[theorem]{Lemma}
\newtheorem{proposition}[theorem]{Proposition}
\newtheorem{remark}[theorem]{Remark}
\newtheorem{definition}[theorem]{Definition}
\newtheorem{corollary}[theorem]{Corollary}

\def\Id{\text{\rm Id}}

\DeclareMathOperator{\esssup}{esssup}
\DeclareMathOperator{\essinf}{essinf}

\title{Optimal CLT error bounds in the second order Wasserstein metric via Stein's method for weakly dependent sequences}

\author{
	Yeor Hafouta\thanks{Department of Mathematics, Ben-Gurion University  of the Negev, Israel. \\ Email: yeorh@bgu.ac.il}
	\qquad
	Juho Lepp\"{a}nen\thanks{Department of Mathematics, Tokai University, Japan. \\
	Email: leppanen.juho.heikki.g@tokai.ac.jp}
}

\begin{document}
	
	\maketitle

    \begin{abstract}
    We use Stein's method to prove general upper bounds on the second order Wasserstein distance in the CLT for general dependent partial sums $S_N=X_1+\ldots +X_N$. Then we use these estimates to provide optimal $O(N^{-1/2})$ rates for a variety of chaotic dynamical systems and  $\phi$- and $\psi$-mixing sequences.
    Our results for $\psi$-mixing sequences are obtained under the fourth moment condition, which matches the best known results in the iid case. This is
     the first time that optimal CLT rates of any kind are obtained for unbounded weakly stationary variables under general mixing conditions. We then apply this result and obtain optimal CLT rates $O\left((\text{Var}(S_N))^{-1/2}\right)$ for non-stationary $\psi$-mixing sequences without any growth rates on $\text{Var}(S_N)$.
    We also obtain close to optimal CLT rates for nonconventional sums and some classes of random dynamical systems with non-uniform decay of correlations.
\end{abstract}

   \section{Introduction}
Let $X_1,X_2,...$ be a sequence of square integrable random variables and let $S_N=X_1+\ldots+X_N$. Then the classical central limit theorem (CLT) states that, when $X_i$ are iid and have zero mean, $\frac{S_N}{\|S_N\|_{L^2}}$ weakly converges towards the standard normal distribution. In the past century there have been many extensions of the CLT for classes of weakly dependent random variables $X_j$. 

Once the CLT is obtained it is natural to study convergence rates.
In the seminal works of Berry \cite{Berry} and Esseen \cite{Esseen}, optimal uniform CLT rates $O(N^{-1/2})$ were obtained for iid $X_i$ under the assumption that $\mathbb E[|X_i|^3]<\infty$. Ever since, uniform CLT rates were obtained  by many authors for weakly dependent sequences generated by sufficiently fast mixing Markov chains or chaotic dynamical systems, and we refer to \cite{HH} for a general (spectral) approach.
Another notable method for proving CLT rates is the, so-called, Stein's method \cite{CGS11,Ross}. However, it is still unclear in general when it yields optimal CLT rates. Recently, uniform (multidimensional) optimal rates for convex sets were  achieved in \cite{LeppSD} for several classes of dynamical systems with linear growth of $\text{Var}(S_N)$. In this paper we will address this problem and show that 
Stein's methods yields optimal CLT rates in the second order Wasserstein metric $\cW_2$ (defined below)  for a variety of weakly dependent sequences (not necessarily stationary).

Recall next that for $1\leq p<\infty$ the $p$-th Wasserstein distance between probability laws $\mu,\nu$ with finite $p$-th moment is given by 
$$
\mathcal W_p(\mu,\nu)=\inf_{(W,Z)}\|W-Z\|_{L^p}
$$
where the infimum runs over all pairs of random variables $(W,Z)$ such that $W$ is distributed according to $\mu$ and $Z$ is distributed according to $\nu$. This is a particular case of a minimal distance, whose study was initiated by P. L\'evy. 
 We refer to \cite[Section 1]{R09} for a detailed discussion about the history of CLT rates in the metric $\cW_p$. In this setting the problem is to determine whether 
$$
\lim_{N\to\infty}\cW_p\left(\mathcal L(S_N/\|S_N\|_{L^2}),N(0,1)\right)=0
$$
and at which rate. Here $N(0,1)$ denotes the standard normal law and $\cL(W)$ is the law of a random variable $W$. 
Note that even in for iid $X_i$ optimal CLT rates in the metric $\cW_p$ were only obtained relatively recently, see \cite{B18} and \cite{R09}. Note that in the iid case optimal $\cW_p$ rates were obtained when $X_i\in L^{p+2}$.

In this paper we will focus our attention to optimal CLT rates for weakly dependent sequences $X_i$ in the second order Wasserstein metric $\mathcal W_2(\cdot,\cdot)$.
We will first use an adaptation of Stein's method to provide general bounds on the $\mathcal W_2$-distance from normality. 
These bounds generalize those in \cite{F19} from the locally
dependent case to sequences with a general dependence structure.
We apply these bounds to obtain optimal rates for dynamical systems
satisfying functional correlation bounds with sufficiently fast
polynomial decay, as well as for $\phi$- and $\psi$-mixing sequences.
In all of these examples, this is the first time that optimal rates are obtained in this metric under our assumptions. 
We also obtain close to optimal CLT rates in $\mathcal W_2$
for random intermittent maps and random nonuniformly expanding
maps on the unit interval. These systems exhibit nonuniform
polynomial decay of correlations.
Moreover, we prove CLT rates for nonconventional sums,
generalizing the results in \cite[Chapter 1]{HK} from estimates
in $\mathcal W_1$ to estimates in $\mathcal W_2$.

For  uniformly bounded $\phi$-mixing sequences with linear growth of the variance of  $S_N$ our results extend the optimal uniform CLT rate of Rio \cite{Rio}  to the metric $\mathcal W_2$. We would also like to refer to \cite{Aust} for optimal rates for  stationary \textbf{interlaced} $\alpha$-mixing sequences and to \cite{DMR23} for optimal rates for   bounded functionals of stationary $\alpha$-mixing sequences with values on a Polish space. Our results show that for $\phi$-mixing sequences with linear growth of the variance (e.g. weakly stationary) strong stationarity can be dropped.
Moreover, for $\psi$-mixing sequences our results hold true under the fourth moment condition $\sup_j\|X_j\|_{L^4}<\infty$. For the best of our knowledge, this is the first time that optimal CLT rates of any kind are obtained for unbounded variables solely under general mixing conditions. 
As an application we are able to prove optimal CLT rates for uniformly bounded $\psi$-mixing sequences without any growth rates on the variance of $S_N$. This is the first time that optimal CLT rates are obtained only under general mixing conditions, without Markovian or dynamical structure like in \cite{DH23,DH25}.

Convergence rates in $\mathcal W_p$ for the CLT and the weak
invariance principle (WIP) have been studied previously for both
autonomous and nonautonomous dynamical systems.
In \cite[Section 3.2]{DMR23}, the optimal $\mathcal W_2$ rate
$O(N^{-1/2})$ in the CLT was established for H\"older observables
of the Liverani--Saussol--Vaienti (LSV) map $T_\eta$
(see \eqref{eq:lsv} for the definition) with parameter $\eta<1/4$.
This is precisely the regime in which the correlation decay rate
$\rho(n)=n^{1-1/\eta}$ satisfies
$$\sum_{n=1}^{\infty}n^2\rho(n)<\infty.$$
Our Theorem \ref{thm:lsv} extends this estimate to nonautonomous
compositions of LSV maps.
For the WIP, $\mathcal W_p$ rates for nonuniformly expanding
and hyperbolic systems admitting Young towers were obtained
in \cite{DMR24,LW24-1}.
The rates in \cite{LW24-1} depend on the integrability of the
return time; for LSV maps, they approach $O(N^{-1/4})$
as $\eta\to0$.
Under a finite fourth-moment assumption on the return time,
\cite{DMR24} obtained the rate
$O(N^{-1/4}(\log N)^{1/4})$ in $\mathcal W_2$.
In the nonautonomous setting, \cite{LW24-2} established
$\mathcal W_p$ rates in the WIP for sequential uniformly
expanding systems.
Quenched rates in the same metric for random Young towers,
including random LSV maps, were obtained recently
in \cite{LSVW25}.
Optimal $\mathcal W_p$ rates in the CLT for sequences of
uniformly expanding maps were established in \cite{DH25}.
To our knowledge, this is the only previous work establishing
such rates for $p>1$ in the CLT for nonautonomous dynamical systems.

The paper is organized as follows.
In Section \ref{Sec2}, we prove general estimates on the
$\cW_2$-distance from normality for a general sum, using an
appropriate metric on the index set. This generalizes the notion
of local dependence, allowing us to consider dependencies between
indices that are far apart.
In Section \ref{Sec3.1.1}, we apply these estimates under
functional correlation decay to obtain optimal CLT rates for
classical sums arising from various dynamical systems.
In Section \ref{Sec3.1.2}, we apply the same estimates to obtain
close to optimal CLT rates for nonconventional sums.
In Section \ref{Sec3.2}, we provide a dynamical scheme for
verifying functional correlation bounds and discuss applications
to stationary and nonstationary dynamical systems using
estimates established in previous works.
Using this scheme, we prove nonuniform functional correlation
bounds for random LSV maps in Section \ref{Sec3.3} and for
random piecewise non-uniformly expanding maps that are only
expanding on average in Section \ref{Sec3.4}.
For both classes of systems, we then obtain close to optimal
$\cW_2$ rates in the CLT.
Section \ref{Sec4} is devoted to optimal CLT rates in $\cW_2$ for $\phi$-mixing uniformly bounded sequences with linear growth of $\text{Var}(S_N)$. In Section \ref{Sec5} we achieved the same task for $\psi$-mixing sequences satisfying $\sup_j\|X_j\|_{L^4}<\infty$ (suboptimal rates are obtained when $\sup_j\|X_j\|_{L^3}<\infty$). As an application, we also obtain optimal $\cW_2$ rates $O\left((\text{Var}(S_N))^{-1/2}\right)$ for uniformly bounded $\psi$-mixing sequences without limitations on the variance growth.
	
	\section{A general Wasserstein-2 bound for the CLT}\label{Sec2}
	
	For two Borel probability measures $E_1$ and $E_2$ on $\bR$, the Wasserstein-$p$ distance, $p \ge 1$, 
	is defined by 
	$$
	\mathcal{W}_p(E_1, E_2) = \biggl( \inf_{\pi \in \Gamma(E_1, E_2) }  \int_{\bR^2} |x-y|^p \, d \pi(x,y)  \biggr)^{ \frac1p },
	$$
	where $ \Gamma(E_1, E_2) $ denotes the set of all Borel probability measures on $\bR^2=\bR \times \bR$ with $E_1$ and $E_2$ as marginals. Note that 
	$$
	p \le q \implies \cW_p(E_1, E_2) \le  \cW_q(E_1, E_2).
	$$
	It was established in \cite{BF81} that, for Borel probability measures $E_0, E_1, E_2, \ldots$ with $\int_{ \bR } |x|^p \, d E_n < \infty$ for $n = 0,1,2,\ldots$, we have 
	$$
	\lim_{n \to \infty} \cW_p( E_n, E_0 ) = 0 \iff 
	E_n \stackrel{\cD}{\to} E_0 \quad \text{and} \quad 
	\lim_{n \to \infty} \int_{\bR} |x|^p \, d E_n(x) = \int_{\bR} |x|^p \, d E_0(x),
	$$
	where $\stackrel{\cD}{\to}$ denotes convergence in distribution. 
	We write $\cL(X)$ for the law (distribution) of a random variable X, and denote by 
	$N(0,1)$ the normal distribution with mean zero and unit variance.

	Let $N \ge 1$, and let $Y_0,\ldots,Y_{N-1}$ be centered
	$\bR$-valued random variables on a probability space $(S,\cA,P)$
	such that
	$
	E(|Y_n|^3) < \infty
	$
	for every $0 \le n < N$. Set
	$
	W=\sum_{n=0}^{N-1}Y_n,
	$
	and assume that
	\begin{align}\label{eq:w_var}
	E(W^2) =1.
	\end{align}
	For any integrable random variable $X$, we write
	$$
	\overline{X} = X - E(X).
	$$
Let $\mathbb N_0=\mathbb N\cup\{0\}$ and let $d:\mathbb N_0\times\mathbb N_0\to\mathbb N$ be a nonnegative symmetric function such that for all $n<N$,
\begin{equation}\label{d cond}
\{0\leq i<N: d(i,n)\geq N\}=\emptyset.    
\end{equation}
    Next, given integers $0 \le n < N$ and $m,\ell \in \bZ$, we define the following auxiliary random variables:
	\begin{align}\label{eq:punctured_sums}
		&W_{n,m} = \sum_{ \substack{ d(i,n) > m \\ 
				0 \le i < N } }  Y_i, \quad 
		Y_{n,m} = \sum_{ \substack{ d(i,n) = m \\ 
				0 \le i < N} } Y_{i}, \quad 
		Y_{n, \langle m, \ell \rangle}  =  \sum_{ 
			\substack{ m < d(i,n) \le \ell \\ 
				0 \le i < N } 
		} Y_i.
	\end{align}
	Note that the following identities hold:
	\begin{align}\label{eq:relations_w}
		\begin{split}
			W_{n, N - 1} &= 0, \quad W_{n, -1} = W,\\
			Y_{n,m} &= W_{n, m - 1} - W_{n, m }, \quad 
			W - W_{n,m}  = \sum_{i \in [n]_m} Y_i,  \\
			Y_{n, \langle m, \ell \rangle} &= 
			W_{n, m} - W_{n, \ell} \quad (m \le \ell),
		\end{split}
	\end{align}
	where 
	$
	[n]_m = \{ 0 \le i < N \: : \: d(i,n) \le m\}.
	$
	
	Given $C_0 > 0$, let $\mathfrak{F}_2(C_0)$ denote the class of all functions 
	$f : \bR \to \bR$ whose second derivative 
	$f''$ is $C_0$-Lipschitz continuous. 
	For convenience, we introduce the following notation for  $f \in \mathfrak{F}_2(C_0)$:
	\begin{align*}
		\delta_f(x,y) = f''(x) - f''(y).
	\end{align*}
	
	\begin{proposition}\label{prop:general_w2} There exist absolute constants 
		$C,C_0, C_1 > 0$ and  for each $f \in \fF$ a
		measurable function $\theta_f : \bR^2 \to \bR$ with $| \theta_f(x,y) | \le C_1 |x-y|^2$
		such that 
		\begin{align}\label{eq:bound_prelim}
			&\cW_2(\cL(W), N(0,1))
			\le C \sup_{f \in \mathfrak{F}_2(C_0)} \biggl\{ 
			|\zeta| +  \sum_{i=1}^{6}  \sqrt{|\cE_i|} \biggr\},
		\end{align}
		where
				\begin{align}\label{eq:xi}
			\begin{split}
				\zeta &= E \biggl\{  
				\sum_{n=0}^{N-1} 
				\sum_{m=1}^{N-1} \sum_{k=2}^{N-1} 
				Y_{n, \langle m(k-1), mk \rangle}  Y_{n,m} Y_n
				+ 
				\sum_{n=0}^{N-1} \sum_{k=1}^{N-1} 
				Y_{n,k}  
				(Y_n)^2  
				+ \frac12 \sum_{n=0}^{N-1} 
				\sum_{m=0}^{N-1} 
				Y_n (Y_{n,m})^2 \biggr\}, 
			\end{split}
		\end{align}
		and the terms $\cE_i$ are defined as follows:
		\begin{align*}
			\cE_1 =\int_0^1
			E \biggl\{ \sum_{n=0}^{N-1} 
			\sum_{m=0}^{N-1}   Y_n (Y_{n,m})^2 
			  \delta_f(   W_{n,m} + uY_{n,m}, W_{n,m}   ) (1-u) 
			\biggr\} \, du,
		\end{align*}
		\begin{align*}
			\cE_2 &=  \sum_{n=0}^{N-1} \sum_{m=0}^{N-1}  
			E[ \theta_f(  W_{n,m} , W  ) ] 
			E[ Y_{n,m} Y_n ]  
			 -  \sum_{n=0}^{N-1} 
			\sum_{k=1}^{ N -1} E[  \delta_f(  W_{n,k-1}, W_{n,k}  ) Y_n ]
			E[ Y_n^2 ] \\ 
			&\qquad -
			\sum_{n=0}^{N-1} \sum_{m=1}^{N-1} 
			\sum_{j \in [n]_m}  \sum_{k=2}^{ N -1}   E[  
			\delta_{f}(   W_{n,m(k-1)},  W_{n, mk}  )  Y_j  ]
			E[ Y_n Y_{n,m} ],
		\end{align*}
		\begin{align*}
			\cE_3 &= \int_0^1 E \biggl\{ \sum_{n=0}^{N-1} 
			\sum_{m=1}^{N-1} \sum_{k=2}^{N-1}
			\overline{
				\delta_f(  W_{n,mk}  + u Y_{n, \langle m(k-1), mk \rangle } ,  W_{n, 2mk} ) 
				Y_{n, \langle m(k-1), mk \rangle } }  
			Y_{n,m} Y_n \biggr\} \, du \\ 
			&\qquad+ \int_0^1 E \biggl\{ \sum_{n=0}^{N-1} \sum_{k=1}^{N-1} 
			\overline{ 
				\delta_f( W_{n,k} + u Y_{n,k},  W_{n, 2k} )    Y_{n,k} } 
			Y_n^2 \biggr\} \, du,
		\end{align*}
		\begin{align*}
			\cE_4 &=  E \biggl\{ \sum_{n=0}^{N-1} 
			\sum_{m=1}^{N-1} \sum_{k=2}^{N-1} 
			\sum_{p=3}^{N-1}  \overline{   
			\delta_f( W_{n, mk (p-1)}, W_{n, mkp}  )	
				 }
			Y_{n, \langle m(k-1), mk \rangle}  Y_{n,m} Y_n \biggl\} \\
			&\qquad +  E \biggl\{ \sum_{n=0}^{N-1} \sum_{k=1}^{N-1} \sum_{p=3}^{N-1} 
			\overline{
				\delta_f( W_{n, k(p-1)},  W_{n, kp}  )
				 }   Y_{n,k}  
			(Y_n)^2  \biggl\}, \\
			&\qquad + E \biggl\{  \sum_{n=0}^{N-1} 
			\sum_{m=1}^{N-1} \sum_{k=2}^{N-1} 
			E[ \delta_f(  W_{n, 2mk} , W )  ]
			Y_{n, \langle m(k-1), mk \rangle}  Y_{n,m} Y_n \biggr\}, \\ 
			&\qquad + E \biggl\{
			\sum_{n=0}^{N-1} \sum_{k=1}^{N-1} 
			E [  \delta_f ( W_{n, 2k}, W  )]
			  Y_{n,k}  
			Y_n^2  \biggr\},
		\end{align*}
		\begin{align*}
			\cE_5 &= E \biggl\{  \sum_{n=0}^{N-1} 
			\sum_{m=1}^{N-1} \sum_{k=2}^{N-1} 
			\sum_{p=3}^{N-1}
			E \biggl[ \overline{ 
				\delta_f( W_{n, mk (p-1)}, W_{n, mkp}   )
				}
			Y_{n, \langle m(k-1), mk \rangle}  
			\biggr]  
			 Y_{n,m} Y_n  \biggr\}  \\
			&\qquad + E \biggl\{  \sum_{n=0}^{N-1} \sum_{k=1}^{N-1} \sum_{p=3}^{N-1} 
			E \biggl[  
			\overline{ \delta_f( W_{n, k(p-1)}, W_{n, kp} ) }   Y_{n,k}  \biggr]
			Y_n^2 \biggr\},
		\end{align*}
		\begin{align*}
			\cE_6 &=  \frac12 E \biggl\{   \sum_{n=0}^{N-1} 
			\sum_{m=1}^{N-1} \sum_{k=2}^{N-1} 
			\overline{
			\delta_f( W_{n,m (k-1)  } , W_{n,m k  }  )	} Y_n Y_{n,m}^2 
			\biggr\} 
			+ \frac12 E \biggl\{  \sum_{n=0}^{N-1} \sum_{k=1}^{N-1} 
			\overline{ \delta_f( W_{n,k-1} , W_{n,k} )
				} Y_n^3
			\biggr\} \\
			&\qquad + \frac12 E \biggl\{  \sum_{n=0}^{N-1} 
			\sum_{m=0}^{N-1} E[ 
			\delta_f(  W_{n,m } , W  ) ] 
			Y_n (Y_{n,m})^2  \biggl\}. 
		\end{align*}

	\end{proposition}

	\subsection{Proof of Proposition \ref{prop:general_w2}}
	
	Throughout the remainder of this section, $C, C_1, C_2,\ldots$ denote positive absolute constants whose values
	may vary from line to line.

    The proof of Proposition \ref{prop:general_w2} builds on the approach developed in
    \cite{F19}, where a rate of convergence in the $\cW_2$ metric in the CLT was established
    for locally dependent summands. We begin by recalling the definition of the Zolotarev ideal metrics. We then record two key preliminary results needed for the proof: an inequality due to Rio relating the Wasserstein-$p$ metrics to the Zolotarev ideal metrics, and a $\cW_p$ error bound in the CLT for sums of independent random variables.
	
	\begin{definition} For $p > 1$ let $\ell_p = \lceil p \rceil - 1$ 
		and denote by 
		$\Lambda_p$ the 
		family of all $f \in C^{ \ell_p }(\bR, \bR)$
		such that 
		$$|f^{(\ell_p)}(x) -  f^{(\ell_p)}(y)| \le |x - y|^{p - \ell_p} \quad \forall (x,y) \in \bR^2.$$ 
		The Zolotarev ideal metric of order $p$ is defined by
		$$
		Z_p(P, Q) = \sup_{f \in \Lambda_p} \biggl\{  \int_{\bR} f \, d P - 
		\int_{\bR} f \, d Q \biggr\}
		$$
		for any two Borel probability measures $P$ and $Q$ on $\bR$.
	\end{definition}

	\begin{lemma}[Theorem 3.1 of \cite{R09}]\label{lem:rio} 
		For any $p > 1$ 
		there exists a positive constant $C_p$ such that 
		for any Borel probability measures $P$ and $Q$ on $\bR$
		with finite absolute moments of order $p$, 
		\begin{align*}
			\cW_p(P, Q) \le C_p ( Z_p(P, Q) 
			)^{\frac1p}.
		\end{align*}
	\end{lemma}

	\begin{lemma}[Theorem 1.1 of \cite{B18}]\label{lem:bobkov} Let $\Theta_n = 
		\sum_{i=1}^n \xi_i$ be a sum of 
		independent random variables 
		$\xi_1,\ldots, \xi_n$ with 
		$E(\xi_i) = 0$ and $E(\Theta_n^2) = 1$. 
		Then for any real $p \ge 1$,  
		\begin{align*}
			\cW_p ( \cL(\Theta_n), N(0,1)) \le 
			C_p \biggl[ \sum_{i=1}^n 
			E( |\xi_i|^{p+2} ) \biggr]^{\frac1p},
		\end{align*}
		where $C_p$ depends continuously on $p$.
	\end{lemma}
	
	Given $h : \bR \to \bR$, 
	we denote by $\Phi(h)$ the expectation of $h$ with respect to $N(0,1)$. We consider
	the Stein 
	equation for $N(0,1)$:
	\begin{align}\label{eq:se}
		f'(w) - wf(w) = h(w) - \Phi(h),
	\end{align}
    which is defined whenever $\Phi(h)$ is.
	Let $p > 1$ be an integer. It follows from \cite[Lemma 6]{B86} that,  if $h \in \Lambda_p$, then 
	\begin{align}\label{eq:f_h}
		f_h(x) = \int_{-\infty}^x e^{ \frac12 (x^2 - t^2 )  } (h(t)  - \Phi(h) ) \, dt 
	\end{align}
	is a solution to \eqref{eq:se}, $f_h$ is $p$ times differentiable 
	and satisfies 
	\begin{align}\label{eq:f2_lip}
		| f_h^{(p)}(x) - f_h^{(p)}(y) | \le C_p |x-y| \quad \forall x,y \in \bR,
	\end{align}
	where $C_p > 0$ is a constant depending only on $p$. In particular, we have 
	\begin{align}
	\sup_{h \in \Lambda_2}	| E[ h (W) ] - \phi(h)  | \le \sup_{ f \in \fF_2(C_0) } |  E[ f'(W)  - W f(W) ]  |,
	\end{align}
	for some constant $C_0 > 0$.
	
	We derive an expansion for $E[h(W)]$ that serves as a counterpart to \cite[Proposition 3.1]{F19} for weakly dependent sequences, including those arising from expanding/hyperbolic dynamical systems and those satisfying $\phi$ or $\psi$-mixing conditions.
	
	\begin{lemma}\label{prop:dec} Let $h \in \Lambda_2$.
		If $f = f_h$ is the solution \eqref{eq:f_h} to the Stein equation 
		\eqref{eq:se}, then there exists an absolute constants $C> 0$ such that 
		\begin{align}\label{eq:expansion_w}
			|   E[h(W)] - \Phi(h) +  \zeta \Phi(f'') | 
			\le \sum_{i=1}^{6} 
			|\cE_i| + C  |\zeta| \cW_2( \cL(W), N(0,1) ),
		\end{align}
	where $\cE_1,\ldots, \cE_6$ and $\zeta$ are defined as in Proposition \ref{prop:general_w2}.
	\end{lemma}
	
	\begin{proof}
		We denote by $U$ a random variable with uniform distribution on $[0,1]$ such that $U$ is independent of all the 
		other variables. 

		Recall that, for $h \in \Lambda_2$, the function $f = f_h$ in \eqref{eq:f_h} 
		satisfies \eqref{eq:f2_lip} with $p = 2$.
		Applying Taylor's theorem and using \eqref{eq:relations_w}, we obtain
		\begin{align*}
			&f(W_{n,m-1}) - f(W_{n,m}) \\
			&= f'( W_{n,m} ) Y_{n,m} 
			+ (Y_{n,m})^2 
			\int_0^1 f''( W_{n,m} + uY_{n,m}
			) (1-u) \, du \\ 
			&= f'( W_{n,m} ) Y_{n,m} 
			+ (Y_{n,m})^2 
			\int_0^1 \biggl[ f''( W_{n,m} + u Y_{n,m}
			 ) - 
			f''( W_{n,m} ) \biggr] (1-u) \, du \\  
			&+ \frac12 (Y_{n,m})^2 f''( W_{n,m} ).
		\end{align*}
		Since $W_{n, -1} = W$, $W_{n, N-1} = 0$ and 
		$E(Y_n) = 0$, we arrive 
		at the following decomposition: 
		\begin{align*}
			E[ W f(W) ]  
			&=  E \biggl\{  \sum_{n=0}^{N-1}  
			Y_n ( f(W_{n, -1}) - f(W_{n,N-1})  ) \biggr\} 
			=  E \biggl\{  \sum_{n=0}^{N-1} 
			\sum_{m=0}^{N-1} 
			Y_n ( f(W_{n,m-1}) - f(W_{n,m})) \biggr\} \\ 
			&=  E \biggl\{  \sum_{n=0}^{N-1} 
			\sum_{m=0}^{N-1}  
			f'( W_{n,m} ) Y_{n,m} Y_n \biggr\} 
			+ \frac12 
			E \biggl\{
			\sum_{n=0}^{N-1} 
			\sum_{m=0}^{N-1} Y_n (Y_{n,m})^2  
			f''( W_{n,m} )  \biggr\} \\ 
			&\qquad+  E \biggl\{ \sum_{n=0}^{N-1} 
			\sum_{m=0}^{N-1}   Y_n (Y_{n,m})^2 
			\biggl[ f''( W_{n,m} + UY_{n,m}
			 ) - 
			f''( W_{n,m} ) \biggr] (1-U) 
			\biggr\} \\
		&=: S + S' + \cE_1. 
		\end{align*}
		
		\noindent\textbf{Decomposition of
			$S$}: We start by writing
		$$
		S = E \biggl\{  \sum_{n=0}^{N-1} 
		\sum_{m=0}^{N-1} \sum_{i=1}^3 R_i(n,m) \biggr\}, 
		$$
		where 
		\begin{align*}
			R_1(n,m) &=  \overline{f'( W_{n,m} )} Y_{n,m} Y_n, \quad 
			R_2(n,m) =  E[f'( W_{n,m} ) - f'(W) ] Y_{n,m} Y_n,  \\ 
			R_3(n,m) &= E[ f'(W) ] Y_{n,m} Y_n.
		\end{align*}
		
		An application of
		Taylor's theorem combined with 
		\eqref{eq:f2_lip} and 
		\eqref{eq:relations_w} yields
		\begin{align*}
			R_2(n,m) &= 
			E[  (W_{n,m} - W) f''( W_{n,m} ) ] Y_{n,m} Y_n + E[ \theta_f(  W_{n,m} , W  )] Y_{n,m} Y_n \\ 
			&=  -  \sum_{j \in [n]_m} E( f''( W_{n,m} ) Y_j ) Y_{n,m} Y_n
			+ E [  \theta_f(  W_{n,m} , W  ) ] 
			Y_{n,m} Y_n,
		\end{align*}
		where $\theta_f(x,y)$ is a measurable function with $|\theta_f(x,y)| \le C_1 |x-y|^2$ for some absolute constant $C_1 > 0$
		independent of $f \in \fF_2(C_0)$.
		
		Consequently,
		\begin{align*}
			E \biggl\{ \sum_{n=0}^{N-1} \sum_{m=0}^{N-1} R_2(n,m) \biggr\} 
			=  &- E \biggl\{ \sum_{n=0}^{N-1} \sum_{m=0}^{N-1} 
			\sum_{j \in [n]_m}  E( f''( W_{n,m} ) Y_j )  Y_{n,m} Y_n   
			\biggr\} + E_{1}
		\end{align*}
		where
		$$
		E_{1} = E \biggl\{ \sum_{n=0}^{N-1} \sum_{m=0}^{N-1}  
		E[ \theta( ( W_{n,m} - W )^2 )] 
		Y_{n,m} Y_n \biggr\}.
		$$
		Further, using  $E(Y_n) = 0$ and
		$W_{n,k} = 0$ for $k \ge N - 1$, we decompose as follows:
		\begin{align*}
			&E \biggl\{ \sum_{n=0}^{N-1} \sum_{m=0}^{N-1} R_2(n,m) \biggr\} \\
			&= E_1-  E \biggl\{  \sum_{n=0}^{N-1}
			E[ ( f''( W_{n,0} ) - f''(0) ) Y_n ] 
			Y_n^2   \biggr\}   \\ 
			&\qquad-  E \biggl\{ \sum_{n=0}^{N-1} \sum_{m=1}^{N-1} 
			\sum_{j \in [n]_m}  E[  ( f''( W_{n,m} )  - f''(0) ) Y_j  ] Y_{n,m} Y_n   
			\biggr\}   \\ 
			&=  E_1  
			- E \biggl\{  \sum_{n=0}^{N-1} 
			\sum_{k=1}^{ N -1} E[  ( f''( W_{n,k-1} ) - f''(W_{n,k}) ) Y_n ]
			Y_n^2   \biggr\}  \\ 
			&\qquad- E \biggl\{
			\sum_{n=0}^{N-1} \sum_{m=1}^{N-1} 
			\sum_{j \in [n]_m}  \sum_{k=2}^{ N -1}    
			E[ ( f''( W_{n,m(k-1)}) - f''(W_{n, mk} )  )  Y_j  ]
			Y_n Y_{n,m}   \biggr\}   = \cE_2.
		\end{align*}
		
		Recall from \eqref{eq:w_var} that $E(W^2) = 1$. Using this property we have
		\begin{align*}
			E \biggl\{ \sum_{n=0}^{N-1} \sum_{m=0}^{N-1} R_3(n,m)  \biggr\} 
			= E[  E( f'(W) ) W^2  ] =   E( f'(W) ) .
		\end{align*}
		
		We turn to $R_1(n,m)$. First, we first use the trivial identity $\overline{f'(0)} 
		= f'(0) - E( f'(0) )
		= 0$ to write
		\begin{align*}
			&\ E \biggl\{ \sum_{n=0}^{N-1} \sum_{m=0}^{N-1} R_1(n,m)  \biggr\} \\ 
			&=  E \biggl\{ \sum_{n=0}^{N-1} 
			\sum_{m=1}^{N-1} 
			[  \overline{f'( W_{n,m} )} - \overline{f'(0)} ] Y_{n,m} Y_n \biggr\} 
			+  E \biggl\{ \sum_{n=0}^{N-1} [ \overline{f'( 
				W_{n,0} )} - \overline{f'(0)} ]
			Y_n^2 ] \biggr\}  \\ 
			&=  E \biggl\{ \sum_{n=0}^{N-1} 
			\sum_{m=1}^{N-1} \sum_{k=2}^{N-1}
			[ \overline{f'( W_{n,m(k-1)} )}
			- \overline{f'( W_{n,mk} )} ] Y_{n,m} Y_n] \biggr\}  
			\\&+  E \biggl\{  \sum_{n=0}^{N-1} \sum_{k=1}^{N-1} 
			[  \overline{f'( 
				W_{n,k-1} )} - \overline{f'(W_{n,k})} ] 
			Y_n^2  \biggr\} .
		\end{align*}
		Next, 
		using the identity $g(x + a) - g(x) = a \int_0^1  g'( x + at ) \, dt$ together with \eqref{eq:relations_w} on summands 
		of the remaining two terms we obtain
		\begin{align*}
			E \biggl\{\sum_{n=0}^{N-1} \sum_{m=0}^{N-1} R_1(n,m)   \biggr\}  
			&=  E \biggl\{  \sum_{n=0}^{N-1} 
			\sum_{m=1}^{N-1} \sum_{k=2}^{N-1}
			\overline{f''( W_{n,mk} 
				+ U Y_{n, \langle m(k-1), mk \rangle } )
				Y_{n, \langle m(k-1), mk \rangle } } 
			Y_{n,m} Y_n \biggr\}  \notag \\ 
			&\quad + E \biggl\{ \sum_{n=0}^{N-1} \sum_{k=1}^{N-1} 
			\overline{f''( 
				W_{n,k} + UY_{n,k})   Y_{n,k}} 
			Y_n^2  \biggr\}.
		\end{align*}
		By adding and subtracting 
		$f''( W_{n,2mk})$ in the first term, 
		and $f''( W_{n,2k})$ in the second term, we 
		find that
		\begin{align*}
			&E \biggl\{ \sum_{n=0}^{N-1} \sum_{m=0}^{N-1} R_1(n,m)  \biggr\} \\
			&= \cE_3 
			+ E \biggl\{ \sum_{n=0}^{N-1} 
			\sum_{m=1}^{N-1} \sum_{k=2}^{N-1}
			\overline{ f''( W_{n,2mk} )
				Y_{n, \langle m(k-1), mk \rangle } } 
			Y_{n,m} Y_n  \biggl\} 
			+ E \biggl\{ \sum_{n=0}^{N-1} \sum_{k=1}^{N-1} 
			\overline{ 
				f''(W_{n, 2k})   Y_{n,k} } 
			Y_n^2  \biggr\}.
		\end{align*}
		
		In order to deal with the remaining two terms, we perform another centering:
		\begin{align*}
			R&:=E \biggl\{ \sum_{n=0}^{N-1} 
			\sum_{m=1}^{N-1} \sum_{k=2}^{N-1} 
			f''( W_{n, 2mk} ) 
			Y_{n, \langle m(k-1), mk \rangle}  Y_{n,m} Y_n  \biggl\} 
			+ 
			E \biggl\{ \sum_{n=0}^{N-1} \sum_{k=1}^{N-1}  
			f''(W_{n, 2k})   Y_{n,k}  
			Y_n^2    \biggr\} \\ 
			&= E \biggl\{ \sum_{n=0}^{N-1} 
			\sum_{m=1}^{N-1} \sum_{k=2}^{N-1}   \overline{ f''( W_{n, 2mk} ) } 
			Y_{n, \langle m(k-1), mk \rangle}  Y_{n,m} Y_n \biggr\} 
			+ E \biggl\{
			\sum_{n=0}^{N-1} \sum_{k=1}^{N-1}   
			\overline{f''(W_{n, 2k})}   Y_{n,k}  
			Y_n^2 \biggr\}   \\ 
			&+ E \biggl\{\sum_{n=0}^{N-1} 
			\sum_{m=1}^{N-1} \sum_{k=2}^{N-1} 
			E[  f''( W_{n, 2mk} ) ]  
			Y_{n, \langle m(k-1), mk \rangle}  Y_{n,m} Y_n \biggl\} 
			+ 
			E \biggl\{ \sum_{n=0}^{N-1} \sum_{k=1}^{N-1} 
			E [  
			f''(W_{n, 2k})]   Y_{n,k}  
			Y_n^2  \biggr\}.
		\end{align*}
		Exploiting $\overline{f''(0)} = 0$ together with a telescopic summation in the first two terms, 
		and by adding and subtracting $f''(W)$ in the last two terms, we see that 
		\begin{align*}
			R  &=  E \biggl\{ \sum_{n=0}^{N-1} 
			\sum_{m=1}^{N-1} \sum_{k=2}^{N-1} 
			\sum_{p=3}^{N-1}   [ \overline{  
				f''( W_{n, mk (p-1)} ) } 
			- \overline{f''( W_{n, mkp}} )  ]
			Y_{n, \langle m(k-1), mk \rangle}  Y_{n,m} Y_n \biggl\} 
			 \\
			&+ 
			E \biggl\{ \sum_{n=0}^{N-1} \sum_{k=1}^{N-1} \sum_{p=3}^{N-1} 
			[ \overline{f''(W_{n, k(p-1)})}  
			- 
			\overline{f''(W_{n, kp})} ]   Y_{n,k}  
			Y_n^2  \biggl\}   \\ 
			&+ E \biggl\{  \sum_{n=0}^{N-1} 
			\sum_{m=1}^{N-1} \sum_{k=2}^{N-1} 
			E[  f''( W_{n, 2mk} ) - f''(W) ]  
			Y_{n, \langle m(k-1), mk \rangle}  Y_{n,m} Y_n \biggr\} 
			 \\ 
			&+ E \biggl\{ 
			\sum_{n=0}^{N-1} \sum_{k=1}^{N-1} 
			E [  
			f''(W_{n, 2k}) - f''(W)]  Y_{n,k}  
			Y_n^2  \biggr\}  \\
			&+ E \biggl\{  \sum_{n=0}^{N-1} 
			\sum_{m=1}^{N-1} \sum_{k=2}^{N-1} 
			E [ f''(W) ]  
			Y_{n, \langle m(k-1), mk \rangle}  Y_{n,m} Y_n \biggr\}  + 
			E \biggl\{ \sum_{n=0}^{N-1} \sum_{k=1}^{N-1} 
			E[ f''(W)  ]   Y_{n,k}  
			Y_n^2  \biggr\} \\
			&= \cE_4 + R',\\
		\end{align*}
		where 
		\begin{align*}
			R' = E \biggl\{  \sum_{n=0}^{N-1} 
			\sum_{m=1}^{N-1} \sum_{k=2}^{N-1} 
			E [ f''(W) ]  
			Y_{n, \langle m(k-1), mk \rangle}  Y_{n,m} Y_n \biggr\}  + 
			E \biggl\{ \sum_{n=0}^{N-1} \sum_{k=1}^{N-1} 
			E[ f''(W)  ]   Y_{n,k}  
			Y_n^2  \biggr\}.
		\end{align*}
		Let
		\begin{align*}
			\bar{R} 
			= E \biggl\{ \sum_{n=0}^{N-1} 
			\sum_{m=1}^{N-1} \sum_{k=2}^{N-1} 
			E[  f''( W_{n, 2mk} ) 
			Y_{n, \langle m(k-1), mk \rangle} ] 
			Y_{n,m} Y_n  \biggl\} 
			+ 
			E \biggl\{ \sum_{n=0}^{N-1} \sum_{k=1}^{N-1} 
			E [  
			f''(W_{n, 2k})   Y_{n,k} ] \ 
			Y_n^2  \biggl\}.
		\end{align*}
		As in the case of $R$, we obtain
	
		\begin{align*}
			\bar{R} 
			&= E \biggl\{ \sum_{n=0}^{N-1} 
			\sum_{m=1}^{N-1} \sum_{k=2}^{N-1} 
			E[  \overline{ f''( W_{n, 2mk} ) } 
			Y_{n, \langle m(k-1), mk \rangle}  ]
			Y_{n,m} Y_n \biggl\} 
			+ 
			E \biggl\{  \sum_{n=0}^{N-1} \sum_{k=1}^{N-1} 
			E [  
			\overline{f''(W_{n, 2k})}   Y_{n,k} ] 
			Y_n^2   \biggr\} \\
			&=  E \biggl\{ \sum_{n=0}^{N-1} 
			\sum_{m=1}^{N-1} \sum_{k=2}^{N-1} 
			\sum_{p=3}^{N-1}
			E \biggl[ [ \overline{  
				f''( W_{n, mk (p-1)} ) } 
			- \overline{f''( W_{n, mkp}} )  ]
			Y_{n, \langle m(k-1), mk \rangle}  
			\biggr]  
			Y_{n,m} Y_n  \biggr\}   \\ 
			&+ 
			E \biggl\{  \sum_{n=0}^{N-1} \sum_{k=1}^{N-1} \sum_{p=3}^{N-1} 
			E \biggl[  
			[ \overline{f''(W_{n, k(p-1)})}  
			- 
			\overline{f''(W_{n, kp})} ]   Y_{n,k}  \biggr]
			Y_n^2 \biggr\} = \cE_5.
		\end{align*}
		Gathering,
		\begin{align*}
			&E \biggl\{  \sum_{n=0}^{N-1} \sum_{m=0}^{N-1} R_1(n,m) \biggr\} 
			= 
			\cE_3 + R - \bar{R} = \cE_3 + \cE_4 - \cE_5  + R'.
		\end{align*}
		We have established the decomposition 
		\begin{align*}
			S &=  
			E( f'(W) )  + R'
			+ \cE_2 +  \cE_3 + \cE_4 - \cE_5.
		\end{align*}
		
		\noindent\textbf{Decomposition of 
			$S'$}: We decompose $S'$
		using a procedure similar to that for $S$:
		\begin{align*}
			2 S' &=  E \biggl\{ \sum_{n=0}^{N-1} 
			\sum_{m=0}^{N-1} Y_n (Y_{n,m})^2  
			f''( W_{n,m} ) 
			\biggr\} \notag \\ 
			&= 
			E \biggr\{  \sum_{n=0}^{N-1} 
			\sum_{m=1}^{N-1} Y_n (Y_{n,m})^2  
			\overline{f''( W_{n,m  } )}
			\biggr\} + 
			E \biggl\{  \sum_{n=0}^{N-1} 
			Y_n^3  
			\overline{f''( W_{n,0} )}
			\biggr\}  \notag \\ 
			&+ E \biggl\{  \sum_{n=0}^{N-1} 
			\sum_{m=0}^{N-1}  Y_n (Y_{n,m})^2  
			E [ f''( W_{n,m  } ) - f''(W) ]
			\biggr\}   \\
			&+ E \biggl\{ 
			\sum_{n=0}^{N-1} 
			\sum_{m=0}^{N-1} 
			Y_n (Y_{n,m})^2  
			E [f''(W)]  \biggr\}  \notag \\ 
			&=  E \biggl\{   \sum_{n=0}^{N-1} 
			\sum_{m=1}^{N-1} \sum_{k=2}^{N-1}
			Y_n (Y_{n,m})^2  
			\biggl[\overline{f''( W_{n,m (k-1)  } )} 
			- \overline{f''( W_{n,m k  } )} \biggr] 
			\biggr\} \\ 
			&+ E \biggl\{  \sum_{n=0}^{N-1} \sum_{k=1}^{N-1} 
			Y_n^3  
			[ \overline{f''( W_{n,k-1} )} - 
			\overline{f''(W_{n,k})}  ]
			\biggr\}   \\
			&+    E \biggl\{  \sum_{n=0}^{N-1} 
			\sum_{m=0}^{N-1} E[ 
			f''( W_{n,m} ) - f''(W)] 
			Y_n (Y_{n,m})^2  \biggl\}
			+ E \biggl\{ 
			E [f''(W)]  \sum_{n=0}^{N-1} 
			\sum_{m=0}^{N-1} 
			Y_n (Y_{n,m})^2  \biggr\}.  
			\\
			&= 2R'' + 2 \cE_6 ,
		\end{align*}
		where 
		$$
		2R'' = E \biggl\{ 
		E [f''(W)]  \sum_{n=0}^{N-1} 
		\sum_{m=0}^{N-1} 
		Y_n (Y_{n,m})^2  \biggr\}.
		$$
		
		Note that $R' + R'' = E[ f''(W) ]  \zeta$. Combining the foregoing decompositions of 
		$S$ and $S'$ we see that 
		\begin{align*}
			&E[ W f(W) ] 
			- E[f'(W)] =
			E[ f''(W) ]  \zeta
			+ \cE_1 + \cE_2 + \cE_3 + \cE_4
			- \cE_5 + \cE_6.
		\end{align*}
		Finally, recalling that $f = f_h$ solves \eqref{eq:se}, we obtain for any $h \in \Lambda_2$ the decomposition
		\begin{align*}
			&E[ h(W) ] - \Phi(h) = E[ f'(W) - W f(W) ] 
			=
			- E [  f''(W) ] \zeta
			- \cE_1 - \cE_2 - \cE_3 - \cE_4 + \cE_5 - \cE_6.
		\end{align*}
		It follows that 
		\begin{align*}
			|   E [h(W)] - \Phi(h) + \zeta \Phi(f'') | 
			\le  \sum_{i=1}^6 |\cE_i| + | \zeta |
			| E [ f''(W) ] - \Phi (f'') |.
		\end{align*}
		To complete the proof of Lemma \ref{prop:dec}, it suffices to note that
		\begin{align*}
			| \zeta ||  E [ f''(W) ] - \Phi (f'') | \le C
			| \zeta |  \cW_1 ( \cL(W), N(0,1) ) \le C  | \zeta |  \cW_2 ( \cL(W), N(0,1) ).
		\end{align*}
	\end{proof}
	
	\begin{proof}[Proof of Proposition \ref{prop:general_w2}] The upper bound \eqref{eq:bound_prelim} follows from \eqref{eq:expansion_w} and
		Lemmas \ref{lem:rio} and \ref{lem:bobkov} by an argument analogous
		to that in \cite[Sections~3.2.2 and~3.2.3]{F19}. For completeness,
		we provide the details below.

			Let $c_0 = 2^{-1} (9/8)^2 < 1$ and $\beta = 2 \zeta$.
		Note that, if $|\beta| \ge ( c_0 / 2 )^{1/2}$, then \eqref{eq:bound_prelim} holds trivially. So, we may assume that 
		\begin{align}\label{eq:beta_ub}
			|\beta| < (  c_0 / 2 )^{ 1 / 2 }.
		\end{align}
		If $\beta \neq 0$, 
		define $m = \round{ c_0 \beta^{-2} }$. 
		Then $m \ge 1$ by \eqref{eq:beta_ub}.
		Let $\{ Z_i \}_{i=1}^m$ 
		be an i.i.d. sequence of random variables such that:
		\begin{align*}
			P \biggl( Z_1 = - \frac32 \biggr) &= \frac{3}{16} - \frac{ \sqrt{m} \beta }{6}, \quad 
			P \biggl( Z_1 = - \frac12 \biggr) = \frac{5}{16} + \frac{ \sqrt{m} \beta }{2}, \\
			P \biggl( Z_1 = \frac12 \biggr) &= \frac{5}{16} - \frac{ \sqrt{m} \beta }{2}, \quad 
			P \biggl( Z_1 =  \frac32 \biggr) = \frac{3}{16} + \frac{ \sqrt{m} \beta }{6}.
		\end{align*}
		By the definition of $c_0$ and $m$ we have $\sqrt{m} |\beta| / 6 < 3/16$, so the above defines 
		a probability distribution. Moreover, by elementary computations, it can be verified that 
		$$
		E[Z_1] = 0, \quad E[Z_1^2] = 1, \quad E[ Z_1^3 ] = \sqrt{m} \beta, 
		\quad E[Z_1^4] \le C.
		$$
		Let $\Theta_m = m^{-1/2} \sum_{i=1}^m Z_i$ 
		and denote by $\kappa_j ( \cdot )$ the $j$th cumulant of a random variable.
		Then,
		$$
		\kappa_3(\Theta_m) = \beta \quad \text{and} \quad 
		\sum_{i=1}^m \frac{E( Z_i^4 )}{m^2} \le C \frac{1}{m} 
		\le C_1  \frac{1}{ c_0 \beta^{-2} - 1 } \le C_2  \beta^2,
		$$
		where $\beta^2 \le c_0/2$ was used. It follows from \cite[Section 2]{B86} that
		\begin{align}\label{eq:fang_approx_iid}
			\sup_{h \in \Lambda_2} \biggl| E[h(\Theta_m)] - \Phi(h) + 
			\frac{\beta}{2} \Phi ( f''_h ) \biggr| 
			\le C \beta^2,
		\end{align}
		and from Lemma \ref{lem:bobkov} that 
		\begin{align}\label{eq:iid_error}
			\cW_2( \cL(\Theta_m), N(0,1) ) 
			\le C_1  m^{-1/2} \le C_2 |\beta|.
		\end{align}
		If  
		$\beta = 0$ one may take $\Theta_m \sim N(0,1)$ in 
		which case \eqref{eq:fang_approx_iid} 
		and \eqref{eq:iid_error} hold trivially.
		
		Now, Lemma \ref{lem:rio} combined with \eqref{eq:fang_approx_iid} implies
		\begin{align*}
			\cW_2(\cL(W), \cL(\Theta_m))
			&\le C \biggl\{
			\sup_{h \in \Lambda_2}  
			[E(h(W)) - E(h(\Theta_m)) ]  
			\biggr\}^{\frac12} \\ 
			&\le C \biggl\{
			\sup_{ h \in \Lambda_2 }  
			[ E[ h(W) ]  - \Phi(h) +  \zeta \Phi(f_h'') ] + \zeta^2
			\biggr\}^{\frac12}.
		\end{align*}
		Using 
		\eqref{eq:expansion_w}, we obtain
		\begin{align*}
			\cW_2(\cL(W), \cL(\Theta_m))
			\le C   \sup_{f \in \mathfrak{F}_2 } \biggl\{ 
			|\zeta| + 
			\biggl[ \sum_{i=1}^{6} 
			|\cE_i| \biggr]^{\frac12} 
			+  [ |\zeta| \cW_2( \cL(W), N(0,1) ) 
			]^{\frac12}
			\biggr\},
		\end{align*}
		so that, by \eqref{eq:iid_error}, 
		\begin{align*}
			&\cW_2(\cL(W), N(0,1))
			\le \cW_2(\cL(W),  \cL(\Theta_m) ) 
			+ \cW_2( \cL(\Theta_m), N(0,1) ) \\ 
			&\le C  \sup_{f \in \fF_2 } \biggl\{ 
			|\zeta| + 
			\sum_{i=1}^{6} 
			\sqrt{|\cE_i|} 
			+  [ |\zeta| \cW_2( \cL(W), N(0,1) ) 
			]^{\frac12} \biggr\}.
		\end{align*}
		
		The upper bound \eqref{eq:bound_prelim}
		now follows by applying the inequality
		$\sqrt{ab} \le \tfrac{1}{2\ve}a + 
		\tfrac{\ve}{2}b$ with 
		$a = |\zeta|$, $b = \cW_2( \cL(W), N(0,1) )$, 
		and $\ve > 0$ sufficiently small.
		
	\end{proof}

\section{Wasserstein-2 rates in the CLT for (nonstationary) dynamical systems}\label{Sec3}

\subsection{Wasserstein-2 bounds under functional correlation decay}\label{Sec3.1}

For every $N\ge1$, let $X_{0,N},\ldots,X_{N-1,N}$ be centered random variables defined on 
a probability space
$(S_N,\cA_N,P_N)$. We assume that
\begin{align}\label{eq:xn_bound}
| X_{n,N} | \le L_N
\end{align}
for every $0\le n<N$, where $L_N\ge1$.

Let $L \ge 1$. For a function $F : [-L, L]^k \to \bR$, where $k \geq 1$, we define 
\[
\text{Lip}(F) = \max_{1 \leq i \leq k} \sup_{x \in  [-L, L]^k } 
\sup_{  \substack{ t \neq t' \\  t,t' \in [-L,L]  }   } 
\frac{|F(x(t/i)) - F(x(t'/i))|}{|t - t'|}.
\]
Here, $x(t/i) \in [-L,L]^k$ denotes the vector obtained from $x$ by replacing its $i$ component with $t$. 
We say that $F$ is separately Lipschitz continuous if $\text{Lip}(F) < \infty$, and we set 
$$
\|F\|_{\Lip } = \|F\|_\infty + \Lip (F), \quad \Vert F \Vert_{\infty} = \sup_{ x \in [-L,L]^k } | F(x) |.
$$
\noindent\textbf{Assumption 1.} We assume that there exists a constant $\theta>0$
independent of $N$, and a sequence of positive constants
$(C_*(N))_{N\ge1}$ such that the following holds for every $N\ge1$.
Whenever $k \ge 2$, $1 \le l < k$, and 
$0 \le n_1 < \ldots < n_k < N$ are integers, and 
 $F : [-L_N, L_N]^k \to \bR$ is a separately Lipschitz continuous function, 
\begin{align}\label{eq:fcb}
	\begin{split}
	&\biggl|  
	\int_{S_N} H(x,x) \, d P_N(x) - \iint_{S_N \times S_N} H(x,y) \, d P_N(x) \, d P_N(y)  
	\biggr| \\
	&\le C_*(N)  \Vert F \Vert_{ \Lip }  ( n_{l + 1} - n_l )^{ - \theta },
	\end{split}
\end{align}
where $H$ is defined by
\begin{align}\label{eq:H}
	H(x,y) = F( X_{ n_1,N }(x), \ldots, X_{n_l,N}(x), X_{ n_{l + 1},N }(y), \ldots , X_{ n_{k},N }(y) ).
\end{align}

\begin{remark}\label{rmk:fcb_many} By induction, \eqref{eq:fcb} readily extends to the case of
"multiple gaps" as follows. For integers
$1 \le l_1 < \ldots < l_p < k$, set $l_0=0$ and $l_{p+1}=k$,
and define the blocks
$$
B_j=\{n_{l_{j-1}+1},\ldots,n_{l_j}\},
\quad j=1,\ldots,p+1.
$$
Writing
$$
X_{B_j,N}(x)
=(X_{n_{l_{j-1}+1},N}(x),\ldots,X_{n_{l_j},N}(x)),
$$
define
$$
G(x_1,\ldots,x_{p+1})
=F(X_{B_1,N}(x_1),\ldots,X_{B_{p+1},N}(x_{p+1})).
$$
	Then, 
	\begin{align}\label{eq:fcb_many_gaps}
		\begin{split}
			&\biggl| \int_{S_N} G(x,\ldots, x) \, d P_N(x) - \int_{S_N}\cdots\int_{S_N} G(x_1, \ldots, x_{p+1}) \, d P_N(x_1) \ldots   \, d P_N(x_{p+1}) \biggr| \\
			&\le C_*(N)  \Vert F \Vert_{ \Lip }  \sum_{j=1}^p  
			( n_{ l_{ j } + 1 } - n_{ l_j } )^{-\theta}.
		\end{split}
	\end{align}	
    In fact, it follows by induction that the same estimate holds if we group some of the blocks together.  In this case the product measure  is replaced by the product measure corresponding to the new block partition. See the proof of \cite[Corollary 3.3]{YeorStien} or \cite[Corollary 1.3.11]{HK}.
\end{remark}

\subsubsection{Optimal rates for classical sums}\label{Sec3.1.1}

Let 
$$
\tilde{S}_N = \sum_{n=0}^{N-1} X_{n,N}, \quad 
W_N =  V_N^{-1/2} \tilde{S}_N,   \quad
V_N = E_N( \tilde{S}_N^2 ),
$$
provided that $V_N > 0$. Here $E_N$ denotes expectation with respect to $P_N$.

\begin{theorem}\label{thm:w2_sums}
Suppose that Assumption 1 is in force with $\theta > 3$. Then there exists a constant $C_\theta > 0$ 
depending only on $\theta$
such that 
for all $N \ge 1$ with $V_N > 0$, 
\begin{align*}
	\cW_2(  \cL( W_N),  N(0,1)  ) \le  C_\theta  L_N^3 ( C_*(N) + 1  )   N V_N^{-3/2}
	+ C_\theta  L^2_N  \sqrt{  ( C_*(N) + 1 ) N } V_N^{-1}.
\end{align*}
In particular, if $V_N^{-1} = O(N^{-1})$ and $\sup_{N \ge 1} \max\{ L_N, C_*(N) \} < \infty$, then 
$$
\cW_2(  \cL( W_N),  N(0,1)  ) = O(N^{-1/2}).
$$
\end{theorem}

\begin{proof} Let $N \ge 1$ with $V_N > 0$. To simplify the notation, throughout
the proof we write
	$$
	X_j=X_{j,N},\qquad L=L_N, \quad C_* = C_*(N), \quad 
	\quad
	(S,\cA,P)=(S_N,\cA_N,P_N),
	$$
and denote expectation with respect to $P$ by $E$. Take $d(n,m)=|n-m|$.

We introduce the following conventions for vectors with entries in $[-L,L]$.
Write $I_N=\{0,\ldots,N-1\}$. Given
$x=(x_0,\ldots,x_{N-1})\in[-L,L]^N$ and $J\subset I_N$,
let $x_J=(x_j)_{j\in J}$, with coordinates listed in increasing
order of their indices. For integers $a,b$, we write
$$
x_{\le a}=(x_j)_{\substack{j\in I_N\\j\le a}},\quad
x_{[a,b]}=(x_j)_{\substack{j\in I_N\\a\le j\le b}},\quad
x_{\ge b}=(x_j)_{\substack{j\in I_N\\j\ge b}}.
$$
We use analogous notation for strict inequalities and half-open intervals, such as
$x_{<a}$ and $x_{(a,b]}$, and for the random vector
$X=(X_0,\ldots,X_{N-1})$, such as $X_{\le a}$ and $X_{[a,b]}$.
For subsets $J_1,J_2\subset I_N$ such that every element of $J_1$
is smaller than every element of $J_2$, we write
$(x_{J_1},x_{J_2})$ for the concatenation of the corresponding
vectors. The same convention applies to any finite collection
of ordered index sets. Empty vectors are omitted from
concatenations.
Individual coordinates whose indices lie outside $I_N$
are omitted from argument lists.

Throughout the proof, $C_\theta$ denotes a generic constant depending
only on $\theta$, and $C$ denotes a generic absolute constant.
These constants may vary from one occurrence to another.

The upper bound \eqref{eq:bound_prelim} 
with $d(n,m) = |n - m|$
implies
\begin{align}\label{eq:w2_prelim}
	&\cW_2(\cL(W_N), N(0,1))
	\le C \sup_{f \in \mathfrak{F}_2(C_0)} \biggl\{ 
	|\zeta| +  \sum_{i=1}^{6}  \sqrt{|\cE_i|} \biggr\},
\end{align}
where $Y_n$ in the definitions of $\zeta$ and $\cE_i$ are given by 
$$
Y_{n} = V_N^{-1/2} X_{n}.
$$ 
Note that we then have 
$$
W_N = \sum_{n=0}^{N-1} Y_{n} \quad \text{and} \quad E( W_N^2 ) = 1.
$$
It remains to control $|\zeta|, |\cE_1|, \ldots, |\cE_6|$ for
$f \in \fF(C_0)$. This will be done by applying \eqref{eq:fcb_many_gaps} for various choices of $F$.
For integers $n,m,\ell,N$ and 
$x = (x_0,\ldots, x_{N-1}) \in [-L, L]^N$ we write
\begin{align*}
	&s_{n,m} = \sum_{ \substack{ |i - n| > m \\ 
			0 \le i < N } }  x_i, \quad 
	x_{n,m} = \sum_{ \substack{ |i - n| = m \\ 
			0 \le i < N} } x_{i}, \quad 
	x_{n, \langle m, \ell \rangle}  =  \sum_{ 
		\substack{ m < |i - n| \le \ell \\ 
			0 \le i < N } 
	} x_i.
\end{align*}

\medskip 

\noindent\textbf{Bound on $|\zeta|$.} Whenever $0 \le n \le n + m \le n + m + \ell < N$, we have
\begin{align}\label{eq:multiple_corr}
|	E(  X_{n} X_{n+m} X_{n+m+\ell} ) | \le ( C_* + 1  ) L^3 \min \{  \max\{  m, \ell \}^{-\theta}, 1 \}.
\end{align}
Indeed, if $ \max\{  m, \ell \} > 0$ this follows by an application of \eqref{eq:fcb} with 
$F : [ -L, L ]^3 \to \bR$ defined by
$$
F( x_n, x_{n+m}, x_{n+m+\ell} ) =  x_n x_{n+m} x_{n+m+\ell},
$$
and otherwise the bound is a direct consequence of \eqref{eq:xn_bound}.
Consequently, as $\theta > 2$, 
\begin{align*}
|\zeta| &\le 
\sum_{n=0}^{N-1} 
\sum_{m=1}^{N-1} \sum_{k=2}^{N-1} 
| E(Y_{n, \langle m(k-1), mk \rangle}  Y_{n,m} Y_n ) |
+ 
\sum_{n=0}^{N-1} \sum_{k=1}^{N-1} 
| E( Y_{n,k}  
Y_n^2 ) |  
+ \frac12 \sum_{n=0}^{N-1} 
\sum_{m=0}^{N-1} 
|E( Y_n Y_{n,m}^2 ) |  \\
&\le C_\theta (C_*  + 1) L^3  V_N^{-3/2} \biggl\{   \sum_{n=0}^{N-1} 
\sum_{m=1}^{N-1} \sum_{k=3}^{N-1} m ( m(k - 2) + 1 )^{ - \theta  } 
+ \sum_{n=0}^{N-1} 
\sum_{m=1}^{N-1}  m^{ 1 - \theta  } 
+ \sum_{n=0}^{N-1} \sum_{k=1}^{N-1} k^{-\theta } \\
&\qquad + \frac12 \sum_{n=0}^{N-1} 
\sum_{m=1}^{N-1}  m^{ - \theta  } + N
\biggr\} \le  C_\theta (C_*  + 1)  L^3 N V_N^{-3/2}.
\end{align*}

\noindent\textbf{Bound on $\cE_1$.} Given $u \in [0,1]$ and integers $0 \le n,m < N$, we define
\begin{align*}
&F_1(  x_{  \le n - m },  x_n, x_{ \ge n  + m}   ; n,m,u ) 
= x_n x_{n,m}^2 
\delta_f \biggl(  V_N^{-1/2} s_{n,m} +   u V_N^{-1/2} x_{n,m}  , V_N^{-1/2} s_{n,m}   \biggr).
\end{align*}
Using the property that $f''$ is $C_0$-Lipschitz continuous, it is straightforward to verify that
$$
\Vert F_1(  \cdot ; n,m,u ) \Vert_{\Lip} \le C V_N^{-1/2}  L^4,
$$
where $C > 0$ is an absolute constant. Since $E(X_n)=0$, applying
\eqref{eq:fcb_many_gaps} for $m\ge1$
and estimating the terms with $m=0$ directly using the Lipschitz
continuity of $f''$, we obtain
\begin{align*}
	| \cE_1|  &=  \biggl| V_N^{ - 3/2 }  \int_0^1
	 \sum_{n=0}^{N-1} 
	\sum_{m=0}^{N-1}   
	E[ F_1(    X_{  \le n - m },  X_n, X_{ \ge n  + m}  ; n,m,u ) ]
 (1-u) \, du \biggr| \\
 &\le C_\theta  ( C_* + 1) L^4 V_N^{-2}  N  \biggl\{  \sum_{m=1}^{N-1}  m^{ - \theta } + 1  \biggr\}
 \le C_\theta ( C_* + 1)  L^4 V_N^{-2}  N .
\end{align*}

\noindent\textbf{Bound on $\cE_2$.} Similar to \eqref{eq:multiple_corr}, we have 
\begin{align}\label{eq:pair_corr}
	|	E(  X_{n} X_{n+m}  ) | \le ( C_* + 1  ) L^2 \min \{  m^{-\theta}, 1 \}.
\end{align}
Next, for integers $0 \le n < N$, $1 \le m < N$, and $1 \le k < N$, we define
\begin{align*}
	F_{2,1}(  x_{ \le n - k  }, x_n, x_{ \ge n + k}  ; n,k ) = \delta_f(  V_N^{-1/2} s_{n,k-1},  V_N^{-1/2} s_{n,k}  )  x_n,
\end{align*}
and, for $2 \le k < N$ we define
\begin{align*}
F_{2,2}(  
x_{ < n - m(k-1) } , x_{[n-m,n+m]}, x_{ > n + m(k-1) } 
; n,m,k )  =	\sum_{j \in [n]_m}  \delta_{f}(  V_N^{-1/2}  s_{n,m(k-1)},  V_N^{-1/2} s_{n, mk}  )  x_j.
\end{align*}
Then, using $f \in \fF(C_0)$ we obtain
\begin{align*}
	\Vert F_{2,1}(  \cdot  ; n,k )   \Vert_{\Lip} \le  C  L^2  V_N^{-1/2} 
	\quad \text{and} \quad 
		\Vert F_{2,2}(  \cdot  ; n,m,k )   \Vert_{\Lip} \le  C  L^2  V_N^{-1/2} m^2.
\end{align*}
Set 
$$
\cD_N = L^4 (C_* + 1) N V_N^{-2}.
$$
Recall that $\theta_f(x,y) \le C|x-y|^2$.
Applying \eqref{eq:fcb_many_gaps} together with \eqref{eq:pair_corr} and \eqref{eq:xn_bound} gives
\begin{align*}
	|\cE_2| &\le  \sum_{n=0}^{N-1} \sum_{m=0}^{N-1}  
	| E[ \theta_f( W_{n,m} , W  )] 
	E[ Y_{n,m} Y_n ]  | \\
	&\qquad+  V_N^{-1/2} \sum_{n=0}^{N-1} 
	\sum_{k=1}^{ N -1} |E[  F_{2,1}(  X_{ \le n - k  }, X_n, X_{ \ge n + k}  ; n,k )  ]
	E[ Y_n^2 ] |
	\\ 
	&\qquad +
	V_N^{-1/2} \sum_{n=0}^{N-1} \sum_{m=1}^{N-1} 
	|E[   F_{2,2}(  
	X_{ < n - m } , X_{[n-m,n+m]}, X_{ > x + m } 
	; n,m,2 )  ]
	E[ Y_n Y_{n,m} ]| \\
	&\qquad +
	V_N^{-1/2}
	\sum_{n=0}^{N-1} \sum_{m=1}^{N-1}  \sum_{k=3}^{ N -1}   
	|E[  
	F_{2,2}(  
	X_{ < n - m(k-1) } , X_{[n-m,n+m]}, X_{ > x + m(k-1) } 
	; n,m,k ) 
	  ]
	E[ Y_n Y_{n,m} ]|  \\
	&\le   C_\theta  \cD_N \biggl( \sum_{m=1}^{N-1}  m^{2-\theta} + 1 \biggr)
	+ C_\theta \cD_N  \sum_{k=1}^{N-1}  k^{ - \theta  }   
	+ C_\theta \cD_N  \sum_{m=1}^{N-1}  m^{ 2 - \theta } \\
	&\qquad + C_\theta \cD_N
	\sum_{m=1}^{N-1} 
	  \sum_{k=3}^{ N -1}   
	m^2  ( m(k-2) )^{-\theta}  \\
	&\le  C_\theta \cD_N,
\end{align*}
where $\theta > 3$ was used in the last inequality.

\medskip

\noindent\textbf{Bounds on $\cE_3,\ldots,\cE_6$.}
The remaining four terms are estimated in the
same manner as $\cE_1$ and $\cE_2$, by applying
\eqref{eq:fcb_many_gaps} with suitable choices of $F$. To avoid
repeating essentially identical computations, for each term we only
specify the functions used, bounds on their Lipschitz norms, and the
resulting estimates obtained from \eqref{eq:fcb_many_gaps}.

For $\cE_3$, we define
\begin{align*}
	&F_{3,1}(   x_{  < n - m(k-1)  ) }, x_{n-m}, x_n, x_{n+m}, x_{ > n + m(k-1) }   ; n,m,k,u  ) \\
	&=\delta_f (   V_N^{-1/2} s_{n,mk}  + u  V_N^{-1/2} x_{n, \langle m(k-1), mk \rangle } ,  V_N^{-1/2} s_{n, 2mk} ) 
		V_N^{-1/2} x_{n, \langle m(k-1), mk \rangle } 
	x_{n,m} x_n \\
	&\qquad - E[    \delta_f(  W_{n,mk}  + u Y_{n, \langle m(k-1), mk \rangle } ,  W_{n, 2mk} ) 
	Y_{n, \langle m(k-1), mk \rangle }  ] x_{n,m} x_n,
\end{align*}
and 
\begin{align*}
	F_{3,2}(   x_{ \le n - k }, x_n, x_{ \ge n + k}  ; n,k,u ) &= 
		\delta_f(  V_N^{-1/2} s_{n,k} + u V_N^{-1/2}  x_{n,k},  V_N^{-1/2}  s_{n, 2k} )   V_N^{-1/2}  x_{n,k} 
	x_n^2 \\
	&\qquad -  	E[
		\delta_f( W_{n,k} + u Y_{n,k},  W_{n, 2k} )    Y_{n,k} ]
	x_n^2.
\end{align*}
These functions satisfy
\begin{align*}
	\Vert F_{3,1}(  \cdot ; n,m,k,u  ) \Vert_{ \Lip } \le C L^4  V_{N}^{-1} m^2 k 
	\quad \text{and} \quad 
	\Vert F_{3,2}( \cdot ; n,k, u ) \Vert_{\Lip} \le  C L^4 V_N^{-1}k.
\end{align*}
Applying \eqref{eq:fcb_many_gaps} gives
\begin{align*}
	|\cE_3| &= \biggl| \int_0^1   V_N^{-1}  \sum_{n=0}^{N-1} 
	\sum_{m=1}^{N-1} \sum_{k=2}^{N-1}
	 E [ F_{3,1}(   X_{  < n - m(k-1)  ) }, X_{n-m}, X_n, X_{n+m}, X_{ > n + m(k-1) }   ; n,m,k,u  )  ] 
	\, du \\ 
	&\qquad+ \int_0^1  V_N^{-1}  \sum_{n=0}^{N-1} \sum_{k=1}^{N-1} 
	E[  F_{3,2}(   X_{ \le n - k }, X_n, X_{ \ge n + k}  ; n,k,u )   ]
	\, du \biggr|
	\le  C_\theta \fD_N.
\end{align*}

For $\cE_4$, define
\begin{align*}
	&F_{4,1}( x_{ < n - mk(p-1) },  x_{ [ n - mk, n - m(k-1)  )  }, x_{n-m}, x_n, x_{n+m} ,  x_{ ( n + m(k-1), n + mk ]  }, x_{ > n + mk(p-1) }; n,m,k,p )
	\\
	&=   
		\delta_f(   V_N^{-1/2} s_{n, mk (p-1)},   V_N^{-1/2} s_{n, mkp}  )	
	x_{n, \langle m(k-1), mk \rangle}  x_{n,m} x_n 
	\\
	&\quad-
	E[
		\delta_f( W_{n, mk (p-1)}, W_{n, mkp}  )	
	]
	x_{n, \langle m(k-1), mk \rangle}  x_{n,m} x_n
\end{align*}
and
\begin{align*}
	&F_{4,2}( x_{ < n - k(p-1) }, x_{n-k}, x_n, x_{n+k}, x_{  > n + k(p-1) }; n,k,p ) \\
	 &=
		\delta_f(  V_N^{-1/2} s_{n, k(p-1)},   V_N^{-1/2} s_{n, kp}  )
	   x_{n,k}  
	x_n^2 
	- E[
		\delta_f( W_{n, k(p-1)},  W_{n, kp}  ) ]   x_{n,k}  
	x_n^2.
\end{align*}
These functions satisfy
\begin{align*}
	\Vert F_{4,1} (\cdot; n,m,k,p) \Vert_{\Lip} \le  C L^4  V_{N}^{-1/2} m^2 k \quad \text{and} \quad 
	\Vert F_{4,2} (\cdot; n,k,p) \Vert_{\Lip} \le  C L^4  V_{N}^{-1/2} k.
\end{align*}
Applying \eqref{eq:fcb} together with \eqref{eq:multiple_corr} gives
		\begin{align*}
	|\cE_4| &= \biggl| V_N^{-3/2}  E \biggl\{ \sum_{n=0}^{N-1} 
	\sum_{m=1}^{N-1} \sum_{k=2}^{N-1} 
	\sum_{p=3}^{N-1}  
	F_{4,1}( X_{ < n - mk(p-1) },  X_{ [ n - mk, n - m(k-1)  )  }, X_{n-m}, X_n, X_{n+m} ,  \\
	&\qquad X_{ ( n + m(k-1), n + mk ]  }, X_{ > n + mk(p-1) }; n,m,k,p )
	\biggl\} \\
	&\qquad +  
	V_N^{-3/2} E \biggl\{ \sum_{n=0}^{N-1} \sum_{k=1}^{N-1} \sum_{p=3}^{N-1} 
	F_{4,2}( X_{ < n - k(p-1) }, X_{n-k}, X_n, X_{n+k}, X_{  > n + k(p-1) }; n,k,p )  \biggl\} \\
	&\qquad + E \biggl\{  \sum_{n=0}^{N-1} 
	\sum_{m=1}^{N-1} \sum_{k=2}^{N-1} 
	E[ \delta_f(  W_{n, 2mk} , W )  ]
	Y_{n, \langle m(k-1), mk \rangle}  Y_{n,m} Y_n \biggr\} \\ 
	&\qquad + E \biggl\{
	\sum_{n=0}^{N-1} \sum_{k=1}^{N-1} 
	E [  \delta_f ( W_{n, 2k}, W  )]
	Y_{n,k}  
	Y_n^2  \biggr\} \biggr| \le  C_\theta \fD_N.
\end{align*}

Next, for $\cE_5$ we define 
\begin{align*}
	&F_{5,1}(
	x_{ < n - mk(p-1) }, x_{  [  n - mk , n - m(k-1)  ) }, x_{  (  n + m(k-1) , n + mk  ] }, x_{ > n + mk(p-1) }; n,m,k,p
	) \\
	&= 
		\delta_f(  V_N^{-1/2} s_{n, mk (p-1)}, V_N^{-1/2} s_{n, mkp}   )
	x_{n, \langle m(k-1), mk \rangle}  - 
	E[
		\delta_f( W_{n, mk (p-1)}, W_{n, mkp}   )
	]
	x_{n, \langle m(k-1), mk \rangle}  
\end{align*}
and
\begin{align*}
	&F_{5,2}(
	x_{ < n - k(p-1) }, x_{n-k}, x_{n+k}, x_{ >  n + k(p-1) }; n,k,p
	) \\
	&=   \delta_f(  V_N^{-1/2} s_{n, k(p-1)}, V_N^{-1/2} s_{n, kp} )    x_{n,k}  - E[\delta_f( W_{n, k(p-1)}, W_{n, kp} ) ]   x_{n,k}.
\end{align*}
These functions satisfy
\begin{align*}
	\Vert F_{5,1}(\cdot; n,m,k,p) \Vert_{\Lip} \le C L^2  V_{N}^{-1/2} m^2 k, \quad 
	\Vert F_{5,2}(\cdot; n,k,p) \Vert_{\Lip} \le C L^2  V_{N}^{-1/2}  k,
\end{align*}
and \eqref{eq:fcb_many_gaps} combined with \eqref{eq:xn_bound}  yields
\begin{align*}
	|\cE_5| &= \biggl|   V_N^{-1/2} E \biggl\{  \sum_{n=0}^{N-1} 
	\sum_{m=1}^{N-1} \sum_{k=2}^{N-1} 
	\sum_{p=3}^{N-1}
	E \biggl[ 
	F_{5,1}(
	X_{ < n - mk(p-1) }, X_{  [  n - mk , n - m(k-1)  ) }, X_{  (  n + m(k-1) , n - mk  ] }, \\
	&\qquad X_{ > n + mk(p-1) }; n,m,k,p
	)
	\biggr]  
	Y_{n,m} Y_n  \biggr\}  \\
	&\qquad + 
	V_N^{-1/2}
	E \biggl\{  \sum_{n=0}^{N-1} \sum_{k=1}^{N-1} \sum_{p=3}^{N-1} 
	E \biggl[  F_{5,2}(
	X_{ < n - k(p-1) }, X_{n-k}, X_{n+k}, X_{ >  n - k(p-1) }; n,k,p
	)  \biggr]
	Y_n^2 \biggr\} \biggr| \\
	&\le C_\theta \fD_N.
\end{align*}

Finally, for $\cE_6$ we define
\begin{align*}
	&F_{6,1}( x_{ < n - m(k-1) }, x_{n-m}, x_n, x_{n+m}, x_{n + m(k-1) }; n,m,k ) \\
	&= 	
		\delta_f( V_{N}^{-1/2} s_{n,m (k-1)  } , V_{N}^{-1/2} s_{n,m k  }  )	 x_n x_{n,m}^2 
		- 	E[
			\delta_f( W_{n,m (k-1)  } , W_{n,m k  }  )	] x_n x_{n,m}^2 
\end{align*}
and
\begin{align*}
	F_{6,2} ( x_{\le n - k}, x_n, x_{ \ge n + k }; n,k  )
	= \delta_f( V_{N}^{-1/2}  s_{n,k-1} , V_{N}^{-1/2} s_{n,k} ) x_n^3 
	- E[ \delta_f( W_{n,k-1} , W_{n,k} )] x_n^3.
\end{align*}
Then,
\begin{align*}
	\Vert F_{6,1}( \cdot ; n,m,k ) \Vert_{\Lip} \le C L^4  V_{N}^{-1/2} m, \quad 
	\Vert F_{6,2}( \cdot ; n,k ) \Vert_{\Lip} \le C L^4  V_{N}^{-1/2},
\end{align*}
and \eqref{eq:fcb_many_gaps} combined with \eqref{eq:multiple_corr} and $f \in \fF(C_0)$ yields
\begin{align*}
	|\cE_6| &= \biggl| \frac12  V_{N}^{-3/2}  E \biggl\{   \sum_{n=0}^{N-1} 
	\sum_{m=1}^{N-1} \sum_{k=2}^{N-1} 
	F_{6,1}( X_{ < n - m(k-1) }, X_{n-m}, X_n, X_{n+m}, X_{n + m(k-1) }; n,m,k )
	\biggr\} \\
	&\qquad+ \frac12 V_{N}^{-3/2} E \biggl\{  \sum_{n=0}^{N-1} \sum_{k=1}^{N-1} 
	F_{6,2} ( X_{\le n - k}, X_n, X_{ \ge n + k }; n,k  )
	\biggr\} \\
	&\qquad + \frac12 E \biggl\{  \sum_{n=0}^{N-1} 
	\sum_{m=0}^{N-1} E[ 
	\delta_f(  W_{n,m } , W  ) ] 
	Y_n (Y_{n,m})^2  \biggl\} \biggr|
	\le C_\theta \fD_N.
\end{align*}

\noindent\textbf{Conclusion.} Combining \eqref{eq:w2_prelim} with the bounds on $|\zeta|, |\cE_1|, \ldots, |\cE_6|$,
we obtain
\begin{align*}
	\cW_2(\cL(W_N), N(0,1)) &\le C_\theta (C_*  + 1)  L^3 N V_N^{-3/2}
	+ \sqrt{
	C_\theta \fD_N } \\
	&\le C_\theta L^3 (C_*  + 1)   N V_N^{-3/2}
	+ 
		C_\theta L^2 \sqrt{ (C_* + 1 ) N }  V_N^{-1},
\end{align*}
as wanted.
\end{proof}

\subsubsection{Close to optimal rates for nonconventional sums}\label{Sec3.1.2}
Let us take $\bar\ell\in\mathbb N$ and increasing polynomials $q_1(n)\leq q_2(n)\leq\ldots  \leq q_{\bar\ell}(n)$ which take positive integer values on the positive integers. Suppose also that $q_1(n)=n$. Let us further take  bounded Lipchitz functions $F_{n,N}$. 
Suppose that $X_{n,N}$ is defined for all $n\leq q_{\bar\ell}(N)$ and satisfies the functional decay of correlations \eqref{eq:H}. 
Set
$$
S_N = \sum_{n=0}^{N-1}\left(F_{n,N}(X_{q_1(n),N},\ldots, X_{q_{\bar\ell}(n),N})- E_N[F_{n,N}(X_{q_1(n),N},\ldots, X_{q_{\bar\ell}(n),N})]\right),
$$
$$
W_N =  V_N^{-1/2} S_N,   \quad
V_N = E_N( S_N^2 ),
$$
provided that $V_N > 0$. Here $E_N$ denotes expectation with respect to $P_N$.
Let us take some $l_N\in\mathbb N, l_N\geq2$ and define a graph on $\{0,1,\ldots ,N-1\}$ by connecting $k$ and $m$ if 
$$
\min\{|q_i(k)-q_j(m)|: 1\leq i,j\leq \bar\ell\}\leq l_N.
$$
Denote by $d_N$ the resulting distance on this graph. Note that as $q_1(n)=n$ we have that $d_N(\cdot,\cdot)<N$ and so \eqref{d cond} holds.
Then the size of a ball of radius $1$ around $n$ is $O(l_N)$, uniformly in $n$ (see \cite[page 57]{HK}). Therefore, for $k\geq2$ the size of ball of radius $k$ is $O(l_N^k)$, uniformly in $n$.  

\begin{theorem}\label{Noncon}
Suppose that Assumption 1 is in force. Assume also that $\|F_{n,N}\|\leq L_N$ where $\|\cdot\|$ is the Lipchitz norm and $L_N\geq1$.
Then there exists a constant $C_0 > 0$ 
such that 
for all $N \ge 1$ with $V_N > 0$, 
\begin{align*}
	\cW_2(  \cL( W_N),  N(0,1)  ) \le L_N^3C_0(NV_N^{-3/2}+\sqrt{N}V_N^{-1})l_N^3+C_0L_N^2\sqrt{C_*(N)}N^{4}l_N^{-\theta/2}.
\end{align*}
Thus, if $V_N\geq sN$ for some $s>0$, $C_*(N)$ and $L_N$ are bounded in $N$,  choosing $l_N = N^{\frac{9}{6+\theta}}$ we get the  rates
$$
	\cW_2(  \cL( W_N),  N(0,1)  )=O(N^{ -1/2 + 27/(6 + \theta) })
$$
and so as $\theta\to\infty$ we get close to the optimal rates $O(N^{-1/2})$.
\end{theorem}
\begin{remark}
 When $F_{n,N}=F$  does not depend on $n$ and $N$ then $V_N$ indeed grows linearly fast and almost optimal CLT rates where obtained in the uniform (Kolmogorov) distance,  see \cite[Chapter 1]{HK} and \cite{YeorStien}. The above theorem complements these results, and obtains similar estimates in $\cW_2$.
\end{remark}

\begin{proof} Let $N \ge 1$ with $V_N > 0$. To simplify the notation, throughout
the proof we write
	$$
	X_j=X_{j,N},\qquad L=L_N, \quad C_* = C_*(N), \quad 
	\quad
	(S,\cA,P)=(S_N,\cA_N,P_N),\quad F_{n,N}=F_n
	$$
and denote expectation with respect to $P$ by $E$.
We also adopt the conventions for vectors with entries in $[-L,L]$
introduced at the beginning of the proof of Theorem \ref{thm:w2_sums}.

Throughout the proof, we use $C$ to denote a generic absolute constant.
The upper bound \eqref{eq:bound_prelim} implies
\begin{align*}
	&\cW_2(\cL(W_N), N(0,1))
	\le C \sup_{f \in \mathfrak{F}_2(C_0)} \biggl\{ 
	|\zeta| +  \sum_{i=1}^{6}  \sqrt{|\cE_i|} \biggr\}
\end{align*}
for some absolute constant $C > 0$,
where $Y_n$ in the definitions of $\zeta$ and $\cE_i$ are given by
$$
Y_{n} = V_N^{-1/2} \overline{F_n(X_{q_1(n)},...,X_{q_{\bar\ell}(n)})}.
$$ 
Note that we then have 
$$
W_N = \sum_{n=0}^{N-1} Y_{n} \quad \text{and} \quad E( W_N^2 ) = 1.
$$
It remains to control $|\zeta|, |\cE_1|, \ldots, |\cE_6|$ for
$f \in \fF_2(C_0)$. This will be done by applying \eqref{eq:fcb_many_gaps} for various choices of $F$.
Define 
$$
\tilde x_n=(x_{q_k(n)})_{k=1}^{\bar\ell}
$$
and 
$$
\check x_n=\overline{F}_{n}(\tilde x_n).
$$
For convenience, we introduce the following notation for integers $n,m,\ell,N$ and 
$x \in [-L, L]^N$:
\begin{align*}
	&s_{n,m} = \sum_{ \substack{ d_N(i,n) > m \\ 
			0 \le i < N } }  \check x_i, \quad 
	x_{n,m} = \sum_{ \substack{ d_N(i,n) = m \\ 
			0 \le i < N} } \check x_{i}, \quad 
	x_{n, \langle m, \ell \rangle}  =  \sum_{ 
		\substack{ m < d_N(i,n) \le \ell \\ 
			0 \le i < N } 
	} \check x_i.
\end{align*}
Denote also 
$$
\check X_n=V_N^{1/2}Y_n=\overline{F_n(X_{q_1(n)},...,X_{q_{\bar\ell}(n)})}.
$$
\medskip 

\noindent\textbf{Bound on $|\zeta|$.} Whenever   
$$
\min \{d_N(n,k),d_N(n,s) \}\geq 2
$$
we have
\begin{align}\label{eq:multiple_corr_nonconv}
|	E(  \check X_{n} \check X_{k} \check X_{s} ) | \le ( C_* + 1  ) L^3 l_N^{-\theta}.
\end{align}
Indeed, this follows by an application of \eqref{eq:fcb} with 
$F : [ -L, L ]^{3\bar\ell} \to \bR$ defined by
$$
F( (x_{q_j(n)})_{j=1}^{\bar\ell}, (x_{q_j(n+m)})_{j=1}^{\bar\ell}, (x_{q_j(n+m+\ell)})_{j=1}^{\bar\ell} )
$$
$$
=  \overline F_n((x_{q_j(n)})_{j=1}^{\bar\ell}) \overline F_{n+m}((x_{q_j(n+m)})_{j=1}^{\bar\ell}) \overline F_{n+m+\ell}((x_{q_j(n+m+\ell)})_{j=1}^{\bar\ell}).
$$
Now, as $|F_n|\leq L$ when
$$
\min \{d_N(n,k),d_N(n,s) \}<2
$$
we will use that for all $n,k,s$,
\begin{align}\label{eq:multiple_corr_nonconv1}
|	E(  \check X_{n} \check X_{k} \check X_{s} ) | \le CL^3.
\end{align}
Consequently,  using the triangle inequality for $d_N$,
\begin{align*}
|\zeta| &\le 
\sum_{n=0}^{N-1} 
\sum_{m=1}^{N-1} \sum_{k=2}^{N-1} 
| E(Y_{n, \langle m(k-1), mk \rangle}  Y_{n,m} Y_n ) |
+ 
\sum_{n=0}^{N-1} \sum_{k=1}^{N-1} 
| E( Y_{n,k}  
Y_n^2 ) |  
+ \frac12 \sum_{n=0}^{N-1} 
\sum_{m=0}^{N-1} 
|E( Y_n Y_{n,m}^2 ) |  \\
&\le C(C_*  + 1) L^3  V_N^{-3/2}(N^5 l_N^{-\theta}+Nl_N^3+N^2l_N^{-\theta}+N l_N^2+N^3l_N^{-\theta}+Nl_N^2)
\end{align*}
where the second term in the brackets accounts to the case $m=1,k=2$, the fourth to the case $k=1$ and the sixth to the case $m=0,1$.

\medskip

\noindent\textbf{Bound on $\cE_1$.} 
First when $m=0,1$ we simply bound the sums  by $CL^3V_N^{-3/2}Nl_N^2$ using that $\delta_f$ is bounded above by some constant. Consider indexes $(n,m)$ with $m\geq2$.
Given $u \in [0,1]$ and integers $0 \le n,m < N$, we define
\begin{align*}
&F_1(  \tilde x_{ J_{n,m} } ; n,m,u ) 
= \check x_n \check x_{n,m}^2 
\delta_f \biggl(  V_N^{-1/2} s_{n,m} +   u V_N^{-1/2} \check x_{n,m}  , V_N^{-1/2}s_{n,m}   \biggr),
\end{align*}
where 
$$
J_{n,m} = \{ n \}\cup\{0\leq i<N-1: d(i,n)\geq m\}
$$
Using the property that $f''$ is $C_0$-Lipschitz continuous, it is straightforward to verify that
$$
\Vert F_1(  \cdot ; n,m,u ) \Vert_{\Lip} \le CN V_N^{-1/2}  L^4.
$$
Since $E(\check X_n) = 0$, an application of \eqref{eq:fcb_many_gaps} (when $m\geq2$) yields
\begin{align*}
	| \cE_1|  &=  \biggl| V_N^{ - 3/2 }  \int_0^1
	 \sum_{n=0}^{N-1} 
	\sum_{m=0}^{N-1}   
	E[ F_1(  \tilde X_{ J_{n,m} }; n,m,u ) ]
 (1-u) \, du \biggr| \\
 &\le C  ( C_* + 1) L^4 V_N^{-2}(Nl_N^{2}+N^3 l_N^{-\theta}).
\end{align*}

\noindent\textbf{Bound on $\cE_2$.}
First, in the first sum in the definition of $\cE_2$, for $m=0,1$ we just bound the resulting sum by $CL^4V_N^{-2}Nl_N^3$. In the second sum,  for $k=1$ we bound the resulting sum by $CL^3NV_N^{-4}l_N$. In the third sum for $m=1$ and $k=2$ we bound the sum by $CL^3Nl_N^4V_N^{-2}$. Otherwise, we use our mixing condition as explained in what follows.

Similar to \eqref{eq:multiple_corr_nonconv}, we have 
\begin{align}\label{eq:pair_corr_nonconv}
	|	E(  \check X_{n} \check X_{k}  ) | \le ( C_* + 1  ) L^2 l_N^{-\theta}
\end{align}
if $d_N(n,k)\geq 2$.
Next, for integers $0 \le n < N$, $2 \le m < N$, and $2 \le k < N$, we define
\begin{align*}
	F_2(\tilde x_{J_{n,k}} ; n,k ) = \delta_f(  V_N^{-1/2} s_{n,k-1},  V_N^{-1/2} s_{n,k}  ) V_N^{-1/2} x_n,
\end{align*}
and, if $3 \le k < N$, we define
\begin{align*}
\tilde{F}_2(\tilde x_{J_{n,m, k}} ; n,m,k )  =	\sum_{j \in [n]_m}  \delta_{f}(  V_N^{-1/2}  s_{n,m(k-1)},  V_N^{-1/2} s_{n, mk}  ) V_n^{-1/2}  x_j,
\end{align*}
where 
$$
J_{n,m, k} =  [n]_m \cup\{0\leq i<N: d_N(i,n)\geq m(k-1)\} 
$$
Then, using $f \in \fF(C_0)$ we obtain
\begin{align*}
	\Vert F_2(  \cdot  ; n,k )   \Vert_{\Lip} \le  C N  L^2  V_N^{-1} 
	\quad \text{and} \quad 
		\Vert \tilde{F}_2(  \cdot  ; n,k )   \Vert_{\Lip} \le  C L^2N^2  V_N^{-1}.
\end{align*}
Consequently, \eqref{eq:fcb_many_gaps} together with the triangle inequality implies
\begin{align*}
	|\cE_2| &\le  \sum_{n=0}^{N-1} \sum_{m=0}^{N-1}  
	| E[ \theta_f( W_{n,m} , W  )] 
	E[ Y_{n,m} Y_n ]  |
	+ \sum_{n=0}^{N-1} 
	\sum_{k=1}^{ N -1} |E[  \delta_f(  W_{n,k-1}, W_{n,k}  ) Y_n ]
	E[ Y_n^2 ] |
	\\ 
	&\qquad +
	\sum_{n=0}^{N-1} \sum_{m=1}^{N-1} 
	\sum_{j \in [n]_m} 
	|E[  
	\delta_{f}(   W_{n,m},  W_{n, 2m}  )  Y_j  ]
	E[ Y_n Y_{n,m} ]| \\
	&\qquad +
	\sum_{n=0}^{N-1} \sum_{m=1}^{N-1} 
	\sum_{j \in [n]_m}  \sum_{k=3}^{ N -1}   
	|E[  
	\delta_{f}(   W_{n,m(k-1)},  W_{n, mk}  )  Y_j  ]
	E[ Y_n Y_{n,m} ]|  \\
	&\le CL^3V_N^{-2}Nl_N^4+  C L^4 (C_* + 1)   N^6l_N^{-\theta}
	V_N^{-2}
	+  L^4  (C_* + 1) N^5l_N^{-\theta}  V_N^{-2}  \\
	&\qquad + C L^4(C_* + 1) N^6l_N^{-\theta}V_N^{-2} 
	+ C_\theta L^4   (C_* + 1)N^7l_N^{-\theta}V_N^{-2}  \\
	&\le  CL^3V_N^{-2}Nl_N^4+C(C_* + 1) L^4 N^7l_N^{-\theta}V_N^{-2}.
\end{align*}

\medskip

\noindent\textbf{Bounds on $\cE_3,\ldots,\cE_6$.}
The remaining terms four terms are estimated in the
same manner as $\cE_1$ and $\cE_2$, by applying
\eqref{eq:fcb_many_gaps} with suitable choices of $F$. 

For $\cE_3$, in the first sum, when $m=1$ and $k=2$ we just bound the term by $CL^3NV_N^{-3/2}l_N^3$. In the second sum when $k=1$ we bound the sum by $CL^3NV_N^{-3/2}l_N$. For the rest of the terms we define 
\begin{align*}
	&F_{3,1}(  \tilde x_{ \bar{J}_{ n,m,k } }; n,m,k,u  ) \\
	&=\delta_f (   V_N^{-1/2} s_{n,mk}  + u  V_N^{-1/2}\check x_{n, \langle m(k-1), mk \rangle } ,  V_N^{-1/2} s_{n, 2mk} ) 
		V_N^{-1/2} \check x_{n, \langle m(k-1), mk \rangle } 
	\check x_{n,m} \check x_n \\
	&\qquad - E[    \delta_f(  W_{n,mk}  + u Y_{n, \langle m(k-1), mk \rangle } ,  W_{n, 2mk} ) 
	Y_{n, \langle m(k-1), mk \rangle }  ] \check x_{n,m} \check x_n,
\end{align*}
where 
\begin{align*}
	\bar{J}_{n,m,k} =  \{ n\}\cup\{0\leq i<N: d(i,n)\geq  m(k-1)\} 
	\cup (  [n]_m \setminus [n]_{m-1} ),
\end{align*}
and 
\begin{align*}
	F_{3,2}(\tilde x_{ J_{n,k} } ; n,k,u ) &= 
		\delta_f(  V_N^{-1/2} s_{n,k} + u V_N^{-1/2}  \check x_{n,k},  V_N^{-1/2}  s_{n, 2k} )   V_N^{-1/2}  \check x_{n,k} 
	x_n^2 \\
	&\qquad -  	E[
		\delta_f( W_{n,k} + u Y_{n,k},  W_{n, 2k} )    Y_{n,k} ]
	(\check{x}_n)^2.
\end{align*}
These functions satisfy
\begin{align*}
	\Vert F_{3,1}(  \cdot ; n,m,k,u  ) \Vert_{ \Lip } \le C L^4  V_{N}^{-1}N^3 
	\quad \text{and} \quad 
	\Vert F_{3,2}( \cdot ; n,k ) \Vert_{\Lip} \le  C L^4 V_N^{-1}N.
\end{align*}
Applying \eqref{eq:fcb_many_gaps}, using the triangle inequality with the distance $d_N$ and bounding the number of variable by $N$ 
therefore gives
\begin{align*}
	|\cE_3| &= \biggl| \int_0^1   V_N^{-1}  \sum_{n=0}^{N-1} 
	\sum_{m=1}^{N-1} \sum_{k=2}^{N-1}
	 E [ F_{3,1}( \tilde X_{ \bar{J}_{ n,m,k } }; n,m,k,u  ) ] 
	\, du \\ 
	&\qquad+ \int_0^1  V_N^{-1}  \sum_{n=0}^{N-1} \sum_{k=1}^{N-1} 
	E[  F_{3,2}(\tilde X_{ J_{n,k} } ; n,k, u )   ]
	\, du \biggr|
	\le CL^3NV_N^{-2}l_N^4+ C_\theta C_* L^4N^7 V_N^{-2}l_N^{-\theta}.
\end{align*}

Similar arguments show that 
$$
\max(|\mathcal E_4|, |\mathcal E_5|,|\mathcal E_6|)\leq CL^4 C_*N^8l_N^{-\theta}V_N^{-2}+CL^4NV_N^{-2}l_N^3.
$$
We omit the details to avoid repetition.
\end{proof}

\subsection{A dynamical scheme for verifying \eqref{eq:fcb}}\label{Sec3.2}
We construct a dynamical scheme to streamline the verification of \eqref{eq:fcb}, which we apply in the subsequent sections to random LSV maps and random piecewise expanding maps.

We assume that $(M_k,d_k)$ is a metric space 
equipped with the sigma-algebra of Borel sets $\cB_k$ 
for each $k \ge 1$. Further, we  assume that $(T_n)_{n \ge 1}$ is a sequence 
of measurable maps $T_n : M_n \to M_{n+1}$.
For $n \ge 0$, we set
$$
\cT^n = T_n \circ \cdots \circ T_1,
$$
with the convention that $\cT^0 = \text{id}_{M_1}$.
Let $\mu$ be a Borel probability measure on $M_1$. 

As in Section \ref{Sec3.1}, 
for a  function $F : \prod_{i=1}^k M_{n_i} \to \bR$ where
$n_1, \ldots, n_k \in \bN$, we define
\[
\text{Lip}(F) = \max_{1 \leq i \leq k} \sup_{x \in  \prod_{i=1}^k M_{n_i} } 
\sup_{  \substack{ t \neq t' \\  t,t' \in M_{n_i}  }   } 
\frac{|F(x(t/i)) - F(x(t'/i))|}{ d_{ n_i }(t,t') }.
\]
Here, $x(t/i) \in \prod_{i=1}^k M_{n_i}$ denotes the vector obtained from $x$ by replacing its $i$ component with $t$. 
We say that $F$ is separately Lipschitz continuous if $\text{Lip}(F) < \infty$, and we set 
$$
\|F\|_{\Lip } = \|F\|_\infty + \Lip (F), \quad \Vert F \Vert_{\infty} = \sup_{ x \in   \prod_{i=1}^k M_{n_i}} | F(x) |.
$$

\begin{proposition}\label{prop:scheme} 
	Suppose there exist constants $A_i, \theta_i > 0$ 
	and $r_i \ge 0$
	for $i \in \{1,2,3\}$ such 
	that the following hold.
	Whenever $k \ge 2$, $1 \le l < k$, and $0 \le n_1 < \ldots < n_k < N$ are 
	integers, and $F :  \prod_{i=1}^k M_{n_i + 1} \to \bR$ is
	a bounded
	separately Lipschitz
	continuous, there exists a countable (mod $\mu$) measurable partition $\cQ$ of $M_1$ such that the following hold:
	\begin{itemize}
		\item[(A1)] There exists a decomposition $\cQ = \cQ^{(g)} \cup \cQ^{(b)}$ where the subcollections 
		$\cQ^{(g)}$ and $\cQ^{(b)}$ are disjoint, and 
		\begin{align*}
			\sum_{ B \in  \cQ^{(b)} } \mu(B) \le A_1 n_{ l + 1  }^{ r_1 } ( n_{l + 1} - n_l )^{ - \theta_1 }.
		\end{align*}
		\item[(A2)] For any $B \in \cQ^{(g)}$, 
		\begin{align}\label{eq:a2}
			\sum_{j = 1}^{l} \textnormal{diam}( \cT^{n_j} B) \le A_2 n_{l + 1}^{r_2} ( n_{l + 1} - n_l  )^{ - \theta_2 }.
		\end{align}
		\item[(A3)] For each $B \in \cQ^{(g)}$ there exists $c_B \in B$ such that 
		\begin{align*}
			&\biggl| \sum_{ B \in \cQ^{(g)} } 
			\int_{M_1} \left (\tilde H(c_B, \cdot)-\int_{M_1} H(c_B, y)\, d\mu(y)\right ) \circ \cT^{n_{l+1}}
			\cdot 
			(\mathbf 1_B-\mu(B))\, d\mu \biggr| 
			\\
			&\le A_3 \Vert F \Vert_{ \Lip } n_{l+1}^{r_3} (n_{l+1} - n_l)^{ - \theta_3 },
		\end{align*}
		where $\tilde H$ is the function that satisfies $\tilde H(x,  \cT^{ n_{l + 1} } ( y) ) = H(x,y)$, 
		and 
		$$
		H(x,y) = F( \cT^{n_1}(x), \ldots, \cT^{n_l}(x), 
		 \cT^{n_{l + 1}}(y), \ldots, \cT^{n_k}(y)
		 ).
		$$
	\end{itemize}
	Then, there exists an absolute constant $C > 0$ such that
	\begin{align}\label{eq:fcb-scheme}
		\begin{split}
			&\biggl| \int_{M_1} H(x,x) \, d\mu(x) - \iint_{ M_1^2 } H(x,y) \, d \mu(x) \, d \mu(y) \biggr| \\
			&\le  C \Vert F \Vert_\infty A_1 n_{ l + 1  }^{ r_1 } ( n_{l + 1} - n_l )^{ - \theta_1 } 
			+   C  \Vert F \Vert_\Lip A_2 n_{l + 1}^{r_2} ( n_{l + 1} - n_l  )^{ - \theta_2 } 
			\\
			&+ C A_3  \Vert F \Vert_{ \Lip } n_{l+1}^{r_3} (n_{l+1} - n_l)^{ - \theta_3 }.
		\end{split}
	\end{align}
\end{proposition}

\begin{proof}
	First, we decompose 
	\begin{align*}
		&\int_{M_1} H(x,x) \, d\mu(x) - \iint_{ M_1^2 } H(x,y) \, d \mu(x) \, d \mu(y) \\
		&= \sum_{ B \in \cQ } \int_{B} \biggl[   H(x,x) - \int_{M_1} H(x,y) \, d \mu(y)  \biggr] \, d \mu(x) \\
		&= \sum_{ B \in \cQ^{(g)} } \int_{B} \biggl[   H(x,x) - \int_{M_1} H(x,y) \, d \mu(y)  \biggr] \, d \mu(x) \\
		&+ \sum_{ B \in \cQ^{(b)} } \int_{B} \biggl[   H(x,x) - \int_{M_1} H(x,y) \, d \mu(y)  \biggr] \, d \mu(x) =: \cI_g + \cI_b.
	\end{align*}
	From (A1) we obtain 
	\begin{align*}
		|\cI_b| \le 2 \Vert F \Vert_\infty \sum_{ B \in \cQ^{(b)} } \mu(B) \le  2 \Vert F \Vert_\infty  A_1 n_{ l + 1  }^{ r_1 } ( n_{l + 1} - n_l )^{ - \theta_1 }.
	\end{align*}
	Next,  using (A2) and the Lipschitz continuity of $F$ separately in each coordinate, we obtain for any $x \in B$, $B \in \cQ^{(g)}$, 
	\begin{align*}
		&\biggl|  H(x,x) - \int_{M_1} H(x,y) \, d \mu(y) - \biggl\{    H(c_B,x) - \int_{M_1} H(c_B,y) \, d \mu(y)  \biggr\} \biggr| \\
		&\le 2 \sum_{ j = 1 }^\ell \Lip(F)  d_{ n_j + 1 }( \cT^{n_j}(x), \cT^{n_j}(c_B) )
		\le 2   \Lip(F)  \sum_{ j = 1 }^\ell  \text{diam} ( \cT^{n_j} B )  \\
		&\le 2 \Vert F \Vert_{ \Lip } A_2 n_{l + 1}^{r_2} ( n_{l + 1} - n_l  )^{ - \theta_2 }.
	\end{align*}
	Therefore,
	\begin{align*}
		\cI_g = \sum_{ B \in \cQ^{(g)} } \int_{B} \biggl[   H(c_B,x) - \int_{M_1} H(c_B,y) \, d \mu(y)  \biggr] \, d \mu(x) 
		+ O(  \Vert F \Vert_{ \Lip } A_2 n_{l + 1}^{r_2} ( n_{l + 1} - n_l  )^{ - \theta_2 } ),
	\end{align*}
	where the constant in the error term is absolute. Finally, by (A3), 
	\begin{align*}
		&\sum_{ B \in \cQ^{(g)} } \int_{B} \biggl[   H(c_B,x) - \int_{M_1} H(c_B,y) \, d \mu(y)  \biggr] \, d \mu(x) \\
		&= \sum_{ B \in \cQ^{(g)} } 
		\int_{M_1} \left (\tilde H(c_B, \cdot)-\int_{M_1} H(c_B, y)\, d\mu(y)\right ) \circ \cT^{n_{l+1}}
		\cdot 
		(\mathbf 1_B-\mu(B))\, d\mu \\
		&= O( A_3 \Vert F \Vert_{\Lip} n_{l+1}^{r_3} (n_{l+1} - n_l)^{ - \theta_3 } ).
	\end{align*}
	Gathering the estimates above we obtain \eqref{eq:fcb-scheme}.
\end{proof}

\begin{remark}\label{rem:fcb_ds} 
Suppose that (A1)-(A3) in Proposition \ref{prop:scheme} are in force. Let $(f_n)_{ 0 \le n \le N }$
be a sequence of Lipschitz continuous and bounded functions
$f_n : M_{n+1} \to \bR$ such that 
$\sup_{0 \le n \le N }\Vert f_n \Vert_{ \Lip } \le L_N$ for $L_N \ge 1$.
Define 
$$
X_n = f_n \circ \cT^n, \quad 0 \le n < N.
$$
Let $0 \le n_1 < \ldots < n_k < N$ be a sequence of integers. If 
$F : [-L_N, L_N]^k \to \bR$ is separately Lipschitz continuous, then
the function $\tilde{F} : \prod_{i=1}^k M_{n_i + 1} \to \bR$ defined by
$$
\tilde{F}(x_1, \ldots, x_k) = F( f_{ n_1 }(x_1), \ldots, f_{n_k}(x_k)   )
$$
is separately Lipschitz continuous with 
$$
\Vert \tilde{F} \Vert_{\Lip} \le L_N \Vert F \Vert_{\Lip}.
$$
In particular, we see that \eqref{eq:fcb} holds with
$$
C_*(N) = C A L_N N^r, \quad A = \max\{A_1,A_2,A_3\}, \quad 
r = \max\{  r_1, r_2, r_3 \}, \quad 
\theta = \min \{  \theta_1, \theta_2, \theta_3  \},
$$
where $C > 0$ is an absolute constant.
\end{remark}

Estimates of the form \eqref{eq:fcb-scheme} have been established
for various classes of stationary and nonstationary
chaotic dynamical systems; see
\cite{LNN25, LS17, F22, LNN26, LT26}.
These include both uniform estimates, with $r_1=r_2=r_3=0$,
and nonuniform estimates, with some of the $r_i$ positive.

In particular, under suitable parameter assumptions, uniform
estimates with $\theta_i>3$ for $1\leq i\leq 3$ are available
for Lipschitz observables of sequential
compositions of LSV maps \cite{L17}
(see also Remark \ref{rem:lsv_uniform}), Sinai billiards with
convex scatterers \cite{LS17}, intermittent maps with unbounded
derivatives \cite{LNN25}, and random Lasota--Yorke maps with
holes \cite{LNN26}.
By the previous remark, these systems satisfy
\eqref{eq:fcb} with $\theta>3$ and $C_*(N)$ bounded uniformly in
$N$, provided that the observables have uniformly bounded Lipschitz norms.
Thus, for these systems, Theorem \ref{thm:w2_sums} gives
$$
\cW_2\bigl(\cL(W_N),N(0,1)\bigr)=O(N^{-1/2}),
$$
provided that $V_N\geq cN$ for some $c>0$ and all sufficiently
large $N$.
In Sections \ref{Sec3.3} and \ref{Sec3.4}, we derive
nonuniform estimates of the form \eqref{eq:fcb-scheme}
for two classes of random dynamical systems with weakly dependent
driving processes:
random LSV maps and random piecewise non-uniformly expanding
maps on the unit interval.

\subsection{Application to random LSV maps}\label{Sec3.3}

Let $\sigma : \Omega \to \Omega$ be an invertible and measure preserving
map on a probability space 
$(\Omega, \cF, \bP)$. We assume that $(\sigma, \bP)$ is ergodic.

For $\eta \in (0,1)$, let $T_\eta : M \to M$ be the LSV map with parameter $\eta$ on $M := [0,1]$:
\begin{align}\label{eq:lsv}
T_\eta(x)=\begin{cases}
	x(1+2^\eta x^\eta) & 0\le x<\frac 1 2 \\
	2x-1 & \frac 1 2 \le x\le 1.
\end{cases}
\end{align}
Let $\beta : \Omega \to (0,1)$ be a measurable map and, for 
each each $\omega \in \Omega$, let $T_\omega = T_{\beta(\omega)}$ be the LSV map with 
parameter $\beta(\omega)$. Set 
\begin{align}\label{eq:random_lsv}
	T_\omega^n:=T_{\sigma^{n-1}\omega}\circ \ldots \circ T_{\sigma \omega}\circ T_\omega, \quad \omega \in \Omega,  \ n\in \mathbb N.
\end{align}
In the same manner, denoting by $\mathcal L_\omega \colon L^1(m)\to L^1(m)$ the transfer operator associated to $T_\omega$ and $m$, where $m$ is the Lebesgue measure on $[0,1]$, set
\[
\mathcal L_\omega^n:=\mathcal L_{\sigma^{n-1}\omega}\circ \ldots \circ \mathcal L_{\sigma \omega}\circ \mathcal L_\omega, \quad \omega \in \Omega,  \ n\in \mathbb N.
\]
We also set $T_\omega^0 := \Id$ and $\mathcal L_\omega^0 :=\Id$ for $\omega \in \Omega$.
In the sequel, we assume that
\begin{align}\label{eq:ass_1}
	\overline \beta:=\esssup_{\omega \in \Omega}\beta(\omega)< 1
\end{align}
and that, for some $\gamma \in (0, \overline \beta]$, 
\begin{align}\label{eq:ass_2}
	b_0 := \bP( \beta(\omega) \le \gamma  ) > 0.
\end{align}
Note that, letting $\underline \beta =\essinf_{\omega\in\Omega} \beta(\omega)$, \eqref{eq:ass_2} holds for every  $\underline\beta <\gamma<1$.

We will assume that the random environment is mixing in a suitable sense. To this end, we define 
the $\alpha$-mixing coefficients for the stationary sequence 
$(  \beta \circ \sigma^{i}  )_{i \in \mathbb{Z}}$ as follows:
$$
\alpha(n) = \sup
\{ | \bP(A \cap B) - \bP(A) \bP(B) | \: : \: 
A \in \mathcal{F}_{-\infty, i}, \: 
B \in \mathcal{F}_{i+n, \infty}, \: i \in \bZ
\}
$$
Here, $\cF_{-\infty, i}$ is the sub-sigma-algebra generated by $( \beta \circ \sigma^j )_{j \le i}$ and 
$\cF_{i,\infty}$ is the sub-sigma-algebra generated by $( \beta \circ \sigma^j )_{j \ge i}$.

\medskip 

\noindent\textbf{Assumption.} Let $p \ge 1$. Throughout we assume that $\gamma < 1/2$, 
and that 
\begin{align}\label{eq:mixing_env}
	\alpha(n) = O( n^{ - q  }   ) \quad \text{holds with $q >  \frac{p}{\gamma}  + 1$.}
\end{align}

We define 
$$
\tau_\omega(x) = \inf \{ n \ge 1 \: : \:  T_\omega^n(x) \in Y  \}, \quad Y = [1/2, 1].
$$
Further, we denote by $\cP_\omega$ the canonical partition (mod $m$) of $[0,1]$ into open subintervals
such that $\tau_\omega$ is constant on each $a \in \cP_\omega$. That is, $\cP_\omega$ consists 
of intervals 
$$
I_n^{(0)}(\omega) := (x_{n+1}(\omega),  x_n(\omega) ) \quad \text{and} \quad 
I^{(1)}_n(\omega) := (y_{n+1}(\omega),  y_n(\omega) ), \quad n \ge 1,
$$
where 
\begin{align*}
	x_n(\omega) =   T_\omega^{-n}(1) , \quad y_n(\omega) = 
	\frac{ x_{n-1}(\sigma \omega) + 1 }{2}, \quad n \ge 1.
\end{align*}
Here, the preimages are taken with respect to the left branch of $T_\omega$ and
$x_0(\omega) = 1$. We denote by $\tau_\omega(a)$ the constant value of 
$\tau_\omega$ on $a \in \cP_\omega$. Note that $T_\omega^n$ maps 
$I_n^{(\theta)}(\omega)$ diffeomorphically onto $(1/2,1)$ for $\theta = 0,1$.

Recall from \cite[Proposition 9]{DGTS} that the cocycle of maps $(T_\omega)_{\omega \in \Omega}$ admits a unique random a.c.i.m $\mu$ on $\Omega \times [0, 1]$, which can be identified with a family of probability measures $(\mu_\omega)_{\omega \in \Omega}$ on $[0, 1]$ such that
\[
T_\omega^*\mu_\omega=\mu_{\sigma \omega} \quad \text{for $\mathbb P$-a.e. $\omega \in \Omega$.}
\]
Moreover, $d\mu_\omega=h_\omega \, dm$ with $h_\omega \in \mathcal C_*$ for some $a>1$, where $\mathcal C_*=\mathcal C_*(a)$ is a cone as in~\cite[Section 2.2]{DGTS}. That is, $\mathcal C_*$ consists of  $\phi \in C^0(0, 1]\cap L^1(m)$ such that $\phi \ge 0$, $\phi$ is decreasing, $X^{\beta+1}\phi$ is increasing (where $X$ denotes the identity map), and 
\[
\int_0^x \phi (t)\, dt\le ax^{1- \overline\beta}\int_0^1 \phi (t)\, dt \quad x\in (0, 1].
\]

Throughout the remainder of this section $C, C_1, C_2, \ldots$ denote positive constants
depending only on the random dynamical systems, whose values may vary from line to line.

\begin{lemma} There exist $K_1, K_2 \in L^p(\Omega)$ such that, for $\bP$-a.e. $\omega \in \Omega$,
	\begin{align}\label{eq:tails}
		m( \tau_\omega \ge n ) \le K_1(\omega) n^{-1/\gamma} \quad \text{and} \quad 
		\mu_\omega( \tau_\omega \ge n ) \le K_2(\omega) n^{1 -1/\gamma}.
	\end{align}
	
\end{lemma}

\begin{proof} 	For $\delta \in \{ \sigma, \sigma^{-1} \}$, define 
	$S_n^{\delta}(\omega) = \sum_{i=0}^{n-1} \psi \circ  \delta^i(\omega)$, where  $\psi(\omega) = \mathbf{1}_{  (0, \gamma] } ( \beta(\omega) )$. Then
	$$
	S_n^\sigma(\omega) = \# \{  0 \le j \le n - 1 \: : \:  \beta( \sigma^j \omega ) \le \gamma  \}.
	$$
	For $\ve \in (0,1/2)$, define 
	\begin{align}\label{eq:N_eps}
		N_\ve(\omega) = \max_{ \delta \in \{\sigma, \sigma^{-1} \} } \max \biggl( \biggl\{ n \ge 1 \: : \:  n^{-1} S_n^{\delta }(\omega) \notin [ b_0(1-\ve), b_0(1+\ve) ]  \biggr\} \cup \{0\} \biggr).
	\end{align}
	It follows from \cite[Proposition A4]{DHL25} that for some constant $C > 0$
    depending only on the random dynamical system,
	\begin{align}\label{eq:tails}
		m( \tau_\omega \ge n ) \le  C( 1 + N_\ve(\omega) )^{1/\gamma}  n^{-1/\gamma} \quad \text{and} \quad 
		\mu_\omega( \tau_\omega \ge n ) \le 
		C ( 1 + N_\ve(\omega) )^{1/\gamma - 1} n^{1 -1/\gamma}.
	\end{align}
	Further, as in the proof of \cite[Corollary 27]{DHL25} we see that 
	\eqref{eq:mixing_env} implies 
    \begin{align}\label{eq:Ne_Lp}
        K_1(\omega) := C( 1 + N_\ve(\omega) )^{1/\gamma} \in L^p \quad \text{and}
        \quad 
        K_2(\omega) := C( 1 + N_\ve(\omega) )^{1/\gamma - 1} \in L^p.
    \end{align}
\end{proof}

For $j \ge 1$ and $x \in [0,1]$, we define recursively
$$
\tau_\omega^{(j)}(x) = \inf \{ n > \tau_\omega^{(j-1)} (x) \: : \: T_\omega^n(x) \in Y \},
$$
where $\tau_\omega^{(0)}(x) = 0$. Thus, $\tau_\omega^{(j)}$ is the $j$th return time to $Y$.  Further, 
we define $L_{\omega, n}$ as the number of returns to $Y$ by time $n$:
$$
L_{\omega, n}(x) = \max \{ 0 \le j \le n \: : \:  \tau_\omega^{(j)}(x) \le n \}.
$$
By definition we have $\tau_\omega^{ (L_{\omega, n}) } \le n < \tau_\omega^{ (L_{\omega, n} + 1) }$. Let 
$$
\cJ_\theta = \begin{cases}
	[0,1/2], &\text{if $\theta = 0$,} \\
	[1/2, 1], &\text{if $\theta = 1$.}
\end{cases}
$$
We consider the (mod $m$) partition of $M$ formed by all intervals of the form
\begin{align}\label{eq:cylinder}
	\begin{split}
		C_{\omega, n}^{(\theta)}(\ell_1, \ldots, \ell_{q+1}) &= \{   L_{\omega, n} = q  \} \cap \bigcap_{j=1}^{q+1} \{ \tau_\omega^{(j)}  = \ell_j \} \cap  \cJ_\theta  \\
		&= I_{ \ell_1 }^{(\theta)}(\omega) \cap \bigcap_{j=1}^q  T_\omega^{ - \ell_j } I^{(1)}_{ \ell_{j + 1} - \ell_j } (  \sigma^{\ell_j} \omega  ).
	\end{split}
\end{align}
Here $\theta \in \{0,1\}$, $1 \le \ell_1 < \ldots < \ell_q \le n < \ell_{q+1}$ and $0 \le q \le n$. Moreover, if $q = 0$, then 
$\ell_1 > n$ and $C_{\omega, n}^{(\theta)}(\ell_1) = I^{(\theta)}_{\ell_1}(\omega)$. Let us denote this partition by $\cC_{\omega, n}$.

\begin{lemma} There exists $K_1 \in L^p(\Omega)$ such that, 
	for any $B \in \cC_{\omega,n}$, and for $\bP$-a.e. $\omega \in \Omega$,
	\begin{align}\label{eq:estim_cylinders}
		|B| \le  \max_{ 0 \le k \le n } K_1(\sigma^k \omega)  n^{ - 1 / \gamma }.
	\end{align}
	Moreover, whenever $0 \le \ell < n$ and $B \in \cC_{\omega,n}$, 
	\begin{align}\label{eq:estim_cylinders_image}
		| T_\omega^\ell B | \le  
		\max_{ \ell \le k \le n } K_1(\sigma^{ k } \omega) 
		(n - \ell )^{ - 1 / \gamma  }.
	\end{align}
	
\end{lemma}

\begin{proof} Let $x, x' \in B$ where $B = C_{\omega,n}^{(\theta)}(\ell_1, \ldots, \ell_{q+1})$ is a set of the form \eqref{eq:cylinder}. 
	If $q = 0$, then both \eqref{eq:estim_cylinders} and \eqref{eq:estim_cylinders_image} follow from \eqref{eq:tails}. 
	
	Let $q > 0$. We have that
	\begin{align*}
		|x - x'| \le  \sup_{ \xi \in I_{\ell_1}^{(\theta)}(\omega) }  \frac{1}{ ( T_\omega^{\ell_1} )' \xi  } | T_\omega^{\ell_1}(x) - T_\omega^{\ell_1}(x')   |.
	\end{align*}
	By uniform expansion at returns we have 
	$$
	\sup_{ \xi \in I_{\ell_1}^{(\theta)}(\omega) }  \frac{1}{ ( T_\omega^{\ell_1} )' \xi  }  \le \lambda^{-1}
	$$
	for some $\lambda > 1$ independent of $\omega$. On the other hand, it follows from the bounded distortion estimate in \cite[Proposition 3.12]{KL21}) combined with \eqref{eq:tails} that, for some constant $C > 0$,
	$$
	\sup_{ \xi \in I_{\ell_1}^{(\theta)}(\omega) }  \frac{1}{ ( T_\omega^{\ell_1} )' \xi  }  \le  C | I_{\ell_1}^{(\theta)}(\omega) | \le C K_1(\omega) \ell_1^{-1 / \gamma}.
	$$
	Consequently, 
	\begin{align}\label{eq:estim_ind}
		|x - x'| \le \min\{  \lambda^{-1},   C K_1(\omega) \ell_1^{-1 / \gamma}   \} | T_\omega^{\ell_1}(x) - T_\omega^{\ell_1}(x')   |,
	\end{align}	
	where $T_\omega^{\ell_1}(x), T_\omega^{\ell_1}(x') \in I_{  \ell_{2} - \ell_1 }^{(1)}(  \sigma^{ \ell_1 } \omega  )$. Iterating \eqref{eq:estim_ind} 
	we obtain 
	\begin{align*}
		|x - x'| \le \min\{  \lambda^{-1},   C K_1(\omega) \ell_1^{-1 / \gamma}   \}  \cdot \prod_{j=1}^q  \min\{  \lambda^{-1},   C K_1( \sigma^{ \ell_j } \omega) ( \ell_{j + 1} - \ell_j )^{-1 / \gamma}   \}.
	\end{align*}
	Either $\ell_1 \ge n / (q + 1)$ or $\ell_{j + 1} - \ell_j \ge n / (q+1)$ for some $1 \le j \le q$. Therefore, setting $\ell_0 = 0$, we have 
	\begin{align*}
		|x - x'| \le C  \max_{ 0 \le j \le q  } K_1( \sigma^{ \ell_j} \omega )  \cdot \lambda^{-q}  \biggl(  \frac{n}{ q + 1}   \biggr)^{ - 1 / \gamma  }
		\le C \max_{ 0 \le k \le n  } K_1( \sigma^{ k } \omega ) n^{-1/\gamma}.
	\end{align*}
	This proves \eqref{eq:estim_cylinders}.
	
	For \eqref{eq:estim_cylinders_image}, observe that 
	\begin{align*}
		&T_\omega^\ell ( C_{\omega, n}^{(\theta)}(\ell_1, \ldots, \ell_{q+1} )  ) \\
		&\subset 
		\begin{cases}
			\{  \tau_{ \sigma^\ell \omega }^{(1)} = \ell_{j+1} - \ell, \ldots,  \tau^{ ( q + 1 - j  ) } = \ell_{q + 1} - \ell   \}, & \ell_j \le \ell < \ell_{j+1}, \: 1 \le j \le q, \\
			\{  \tau^{ ( 1) }_{ \sigma^\ell \omega } = \ell_1 - \ell, \ldots,   \tau^{(q + 1)}_{ \sigma^\ell \omega  } = \ell_{q+1} - \ell \}, &\ell < \ell_1,
		\end{cases}
	\end{align*}
	where $\ell_{q} - \ell  \le n - \ell$ and 
	$\ell_{q+1} - \ell > n - \ell > 0$. In particular, $T_\omega^\ell ( C_{\omega, n}^{(\theta)}(\ell_1, \ldots, \ell_{q+1} )  )$ is contained in a set belonging to $\cC_{ \sigma^\ell \omega, n - \ell}$ and \eqref{eq:estim_cylinders_image} follows from \eqref{eq:estim_cylinders}.
\end{proof}

\begin{proposition}\label{prop:fcb_lsv} 
	Assume \eqref{eq:mixing_env}. Let  $F : [0,1]^k \to \bR$ be separately Lipschitz
    continuous, where $k \ge 2$. Let $0 \le n_1 < \ldots < n_k$ and define 
	\begin{align}\label{eq:def_H_lsv}
	H_\omega(x, y) =F(T_{\omega}^{n_1}(x), \ldots, T_{\omega}^{n_l}(x), T_{\omega}^{n_{l+1}}(y), \ldots, T_\omega^{n_k}(y)).
	\end{align}
	Let $\delta > 0$. Then, there exists 
	$\tilde K_\delta \in L^p(\Omega)$ 
    such that for $\bP$-a.e. 
	$\omega \in \Omega$, 
	\begin{align}\label{eq:fcb_lsv}
		\begin{split}
			&\biggl| \int_M H(x,x) \, d\mu_\omega(x) - \iint_{ M^2 } H(x,y) \, d \mu_\omega(x) \, d \mu_\omega(y) \biggr| \\
			&\le     \Vert F \Vert_{ \Lip } \tilde K_{ \delta}(\omega) n_{l + 1}^{  1 / p + \delta   } ( n_{l + 1} - n_l )^{ 1 - 1 / \gamma }.
		\end{split}
	\end{align}
\end{proposition}

\begin{proof} We verify (A1)-(A3) in Proposition \ref{prop:scheme}. Without loss of generality we assume that $n_{l + 1} - n _l > 3$. We define 
	$$
	n_* = n_l + \round{ \frac{ n_{l + 1} - n_l  }{3} }.
	$$
	We let $\cQ_\omega$ be the partition of $M$ consisting of all intervals $C_{\omega, n_*}^{(\theta)}(\ell_1, \ldots, \ell_{q+1})$ as 
	in \eqref{eq:cylinder}. Set 
	$$
	\ell_{\#} = n_l + \round{   \frac{  2 ( n_{l + 1} - n_l ) }{3} }.
	$$
	We decompose $\cQ_\omega = \cQ_\omega^{(g)} \cup \cQ_\omega^{(b)}$ where $\cQ_\omega^{(g)}$ consists of all 
	partition elements $C_{\omega, n_*}^{(\theta)}(\ell_1, \ldots, \ell_{q+1})$  with $\ell_{q + 1} \le \ell_\#$ and 
	$\cQ^{(b)}_\omega$ consists of all 
	partition elements $C_{\omega, n_*}^{(\theta)}(\ell_1, \ldots, \ell_{q+1})$  with $\ell_{q + 1} > \ell_\#$.

    In the remainder of the proof, $C$ denotes a positive constant depending only 
    on the random dynamical system, whose value may vary from line to line.
	
	\medskip 
	
	\noindent\textbf{(A1):} We have
	\begin{align*}
		\sum_{ B \in \cQ_\omega^{(b)} } \mu_\omega(B) 
		\le   \mu_\omega(  \tau_\omega^{ ( L_{ \omega, n_* } + 1 ) }  \ge \ell_\#  ).
	\end{align*}
	Next, observe that 
	$$
	\{  \tau_\omega^{ ( L_{ \omega, n_* } + 1 ) }  \ge \ell_\#   \} \subset   T_\omega^{-n_*}  \{   \tau_{  \sigma^{n_*} \omega  } \ge  \ell_\# - n_*   \}.
	$$
	Let $\delta > 0$. Using \eqref{eq:tails} we obtain 
	\begin{align*}
		\sum_{ B \in \cQ_\omega^{(b)} } \mu_\omega(B)  &\le \mu_{ \sigma^{n_*} \omega }(   \tau_{  \sigma^{n_*} \omega  } \ge  \ell_\# - n_*   ) \le K_2(  \sigma^{n_*} \omega ) ( \ell_\# - n_* )^{ 1 - 1 / \gamma  } \\
		&\le C K_{2, \delta}(  \omega ) n_*^{  1/p + \delta  } ( n_{l+1}  - n_l )^{ 1 - 1 / \gamma  } 
		\le C K_{2, \delta}(  \omega ) n_{l + 1}^{  1/p + \delta  } ( n_{l+1}  - n_l )^{ 1 - 1 / \gamma  },
	\end{align*}
	where $K_{2, \delta} \in L^p$. Here, the second-to-last inequality follows 
    from \cite[Lemma 2]{DH CMP}.
	
	\medskip 
	
	\noindent\textbf{(A2):}  Let $B \in \cQ_\omega^{(g)}$. Since $n_l < n_* \le n_{l + 1}$, using 
	\eqref{eq:estim_cylinders_image} 
	we obtain 
	\begin{align*}
		\sum_{ j = 1 }^{ l }  \text{diam}( T_\omega^{ n_j } B  ) &\le \sum_{j=1}^l C   \max_{ 0 \le k \le n_* } K_1( \sigma^k \omega ) (n_* - n_j)^{-1/\gamma} \\
		&\le C \max_{0 \le k \le n_*} K_1( \sigma^k \omega ) (n_{l + 1} - n_l)^{ 1 - 1 / \gamma } \\
		&\le C K_{1, \delta}(\omega) n_{l+1}^{ 1/p + \delta } (n_{l + 1} - n_l)^{ 1 - 1 / \gamma },
	\end{align*}
	where $K_{1, \delta} \in L^p$.
    
    \medskip
    
	\noindent\textbf{(A3):} Fix any $c_B \in B$ for each $B \in \cQ_\omega^{(g)}$. We have 
	\begin{align*}
		\cI &:= \biggl| \sum_{ B \in \cQ_\omega^{(g)} } 
		\int_M \left (\tilde H(c_B, \cdot)-\int_M H(c_B, y)\, d\mu_\omega(y)\right ) \circ T_\omega^{n_{l+1}}
		\cdot 
		(\mathbf 1_B-\mu_\omega(B))\, d\mu_\omega \biggr| 
		\\
		&\le 2 \Vert F \Vert_\infty  \sum_{ B \in \cQ_\omega^{(g)} } \mu_\omega(B) \Vert \cL_{\omega}^{ n_{ l + 1 } } (   h_\omega - \tilde h_{\omega, B}   ) \Vert_{L^1(m)},
	\end{align*}
	where 
	$$
	\tilde h_{\omega, B} = \frac{ \mathbf{1}_B   h_\omega   }{ \mu_\omega( B )  }.
	$$
	We write 
	$$
	\cL_{\omega}^{ n_{ l + 1 } } (   h_\omega - \tilde h_{\omega, B}   ) =  \cL_{ \sigma^{ \ell_{q + 1} } \omega }^{ n_{l + 1} - \ell_{q + 1}   }  \cL_{\omega}^{ \ell_{  q + 1 } } (   h_\omega - \tilde h_{\omega, B}   ).
	$$
	The density $\cL_{\omega}^{ \ell_{  q + 1 } } (   \tilde h_{\omega, B}  )$ is supported on $Y$ and 
	\begin{align*}
		\cL_{\omega}^{ \ell_{  q + 1 } } (   \tilde h_{\omega, B}  ) (y) =  \frac{  h_\omega( y_{ \omega, B }  )   }{ ( T_\omega^{ \ell_{q + 1} } )' y_{\omega, B}  } 
		\frac{1}{ \mu_\omega(B) },
	\end{align*}
	where $y_{ \omega, B } \in B$ is the unique preimage of $y \in Y$ under $T_\omega^{ \ell_{q + 1} }$. Let us denote the corresponding
    measure by $\nu_{ \sigma^{ \ell_{q+1} } \omega, B  }$:
    $$
    \nu_{ \omega, B }(A) = \int_A \cL_{ \sigma^{ -\ell_{q+1} } \omega}^{ \ell_{  q + 1 } } (   \tilde h_{ \sigma^{ -\ell_{q+1} } \omega, B}  ) \, d m
    $$
    
    Let $y,z \in Y$. Then,
	\begin{align*}
		&| \log  \cL_{\omega}^{ \ell_{  q + 1 } } (   \tilde h_{\omega, B}  ) (y) - \log \cL_{\omega}^{ \ell_{  q + 1 } } (   \tilde h_{\omega, B}  ) (z) | \\
		&\le | \log h_\omega (  y_{\omega, B} )  -   \log h_\omega (  z_{\omega, B} )   | + | \log  ( T_\omega^{ \ell_{q + 1} } )' y_{\omega, B} - 
		\log ( T_\omega^{ \ell_{q + 1} } )' z_{\omega, B} 
		| =: I + II.
	\end{align*}
	Recalling that $h_\omega \ge b > 0$ is Lipschitz continuous on $Y$ with Lipschitz constant independent of $\omega$, we have 
	\begin{align*}
		I \le b^{-1} \Lip(h_\omega |_Y) | y_{\omega, B} - z_{\omega, B} | \le   C b^{-1} \Lip(h_\omega |_Y) |  y - z | \le C |y - z|.
	\end{align*}
	Moreover, using the chain rule we obtain
	\begin{align*}
		&II \le | \log (T_\omega^{ \ell_1 }  )' ( y_{\omega, B} )  -  \log (T_\omega^{ \ell_1 }  )' ( z_{\omega, B} )  |
		\\&+ \sum_{j=1}^q  | \log (T_\omega^{ \ell_{j+1} - \ell_j  }  )' ( T_\omega^{ \ell_j } y_{\omega, B} )  -  
		\log (T_\omega^{ \ell_{j+1} - \ell_j  }  )' ( T_\omega^{ \ell_j } z_{\omega, B} ) 
		|.
	\end{align*}
	The distortion bound in \cite[Proposition 3.12]{KL21} together with uniform expansion at returns then implies
	\begin{align*}
		II \le \sum_{j=1}^{q+1} C |  T_\omega^{ \ell_j } (y_{\omega, B}) -  T_\omega^{ \ell_j } (z_{\omega, B})  | \le  C \sum_{j=0}^{q} \lambda^{ - j } |y - z| \le C |y - z|.
	\end{align*}
	Now, it follows from \cite[Proposition A3]{DHL25}-(c) and \cite[Proposition 3.14]{KL21}
    that $\mu_\omega$ and
    $\nu_{\omega, B}$ are both
    regular
	in the sense of \cite[Definition A2]{DHL25} for $\bP$-a.e. $\omega \in \Omega$, provided that $K$ there is chosen sufficiently large (independent of $\omega$). Moreover, by \eqref{eq:tails},
    $$
    \nu_{\omega, B}( \tau_\omega \ge n ) \le C m( \tau_\omega \ge n  )
    \le 
    C(N_\ve(\omega) + 1 )^{1/\gamma} n^{ - 1 / \gamma}
    $$
    and
    $$
    \mu_\omega( \tau_\omega \ge n ) \le C (N_\ve(\omega) + 1 )^{1/\gamma - 1}  n^{1 - 1 / \gamma}.
    $$
    Consequently, in the terminology of \cite[Definition A2]{DHL25},
    $\mu_\omega$ and $\nu_{\omega, B}$ are regular with the same tail bound
    $$
    r_\omega(n) = C (N_\ve(\omega) + 1 )^{1/\gamma}  n^{1 - 1 / \gamma}.
    $$
    It follows by \cite[equation (A.10)]{DHL25} that 
    \begin{align*}
	&\Vert  \cL_{\omega}^{ n_{ l + 1 } } (   h_\omega - \tilde h_{\omega, B}   )  \Vert_{  L^1(m) } \\
    &\le 
    | ( T_{ \sigma^{ \ell_{q + 1}  } \omega  }^{ n_{ l + 1 } } )_* \nu_{ \sigma^{ \ell_{q + 1}  }  \omega, B} - ( T_{ \sigma^{ \ell_{q + 1}  } \omega  }^{ n_{ l + 1 } } )_* \mu_{ \sigma^{ \ell_{q + 1}  } \omega  }  | \\
    &\le C (N_\ve( \sigma^{ \ell_{q + 1}  } \omega) + 1 )^{1/\gamma}  (
    n_{ l + 1 } - \ell_{q+1}
    )^{1 - 1 / \gamma} = K_3( \sigma^{ \ell_{q + 1}  } \omega ) ( 
    n_{ l + 1 } - \ell_{q+1}
    )^{1 - 1 / \gamma} 
	\end{align*}
    where $|\cdot|$ denotes the total variation of signed measures 
    and $K_3 \in L^p(\Omega)$ by \eqref{eq:Ne_Lp}.
    Hence, for any $\delta > 0$,
    \begin{align*}
	&\Vert  \cL_{\omega}^{ n_{ l + 1 } } (   h_\omega - \tilde h_{\omega, B}   )  \Vert_{  L^1(m) } \le K_{3,\delta}( \omega ) \ell_{q + 1}^{1/p + \delta} ( n_{ l + 1 } - \ell_{q+1} )^{1 - 1 / \gamma},
	\end{align*}
    where $K_{3, \delta} \in L^p(\Omega)$.
    Finally, since $\ell_{q + 1} \le \ell_\#$, we conclude that 
	\begin{align*}
		|\cI| \le C \Vert F \Vert_{\Lip} K_{3, \delta}(\omega) n_{l + 1}^{  1 / p + \delta   } ( n_{l + 1} - n_l )^{ 1 - 1 / \gamma },
	\end{align*}
    
    The desired inequality now follows by Proposition \ref{prop:scheme}.
	
\end{proof}

Let $\varphi : \Omega \times M \to \bR$ be a measurable map such that 
\begin{itemize}
	\item $\int_M \varphi_\omega \, d \mu_\omega = 0$ 
	for $\bP$-a.e. $\omega \in \Omega$, 
	where $\varphi_\omega(x) = \varphi(\omega ,x)$, and 
	\item $\omega \mapsto \Vert \varphi_\omega \Vert_{ \Lip } \in L^r(\Omega)$ for some $r \ge 2$.
\end{itemize}
We define  
$$
S_{\omega, N} = \sum_{n=0}^{N-1} \varphi_{ \sigma^n \omega } \circ T_\omega^n, \quad V_{\omega, N} = \mu_\omega( S_{\omega, N}^2 ), 
\quad W_{\omega, N} = V_{\omega, n}^{-1/2} S_{\omega, N},
$$
provided that $V_{\omega, N} > 0$.

\begin{theorem}\label{thm:lsv} Assume that $\gamma < 1/4$. Then, for any $\delta > 0$, 
	\begin{align*}
		\cW_2(  \cL(  W_{\omega, N}  ), N(0,1)  ) \le C_\delta(\omega) 
		\biggl(  N^{ 1 + \frac3p + \frac3r + \delta } V_N^{-3/2} 
		+
		N^{\frac12 + \frac{1}{2p} + \frac{5}{2r} + \delta} V_N^{-1}
		   \biggr),
	\end{align*}    
	where $C_\delta \in L^q$ for $1/3q = 1/r + 1/p$.
\end{theorem}

\begin{proof} Since $\omega \mapsto \Vert \varphi_\omega \Vert_{ \Lip } \in L^r(\Omega)$, 
\cite[Lemma 2]{DH CMP} implies
	$$\max_{  0 \le k \le n - 1 } \Vert \varphi_{ \sigma^k \omega } \Vert_{\Lip} \le \hat K_\delta (\omega) n^{ 1 / r + \delta }$$
    where $\hat K_\delta \in L^r(\Omega)$, for any $\delta > 0$
    . 
    Therefore, Remark \ref{rem:fcb_ds} together with Proposition \ref{prop:fcb_lsv}
    implies that \eqref{eq:fcb} holds for $\bP$-a.e. $\omega \in \Omega$ with
    $$
    C_*(N) = C N^{ 1 / r + 1/p + \delta } \tilde K_\delta (\omega)  \hat K_\delta (\omega)  , \quad 
    \theta = 1/\gamma - 1,
    $$
    where $\delta$ can be taken arbitrarily small and $C > 0$ is independent of 
    $\omega$.
	We obtain the desired estimate by applying Theorem \ref{thm:w2_sums}.
\end{proof}

\begin{remark}\label{rem:lsv_uniform}
	The proof of Proposition \ref{prop:fcb_lsv} also gives
	\begin{align}\label{eq:fcb_slow_lsv}
		\biggl| \int_M H(x,x) \, d\mu_\omega(x)
		- \iint_{M^2} H(x,y) \, d\mu_\omega(x) \, d\mu_\omega(y) \biggr|
		\le C_* \Vert F \Vert_{\Lip}
		(n_{l+1}-n_l)^{1-1/\overline{\beta}},
	\end{align}
	where $H$ is defined as in \eqref{eq:def_H_lsv} and $C_*>0$
	depends only on the cone $\cC_*(a)$ and $\overline{\beta}$.
	This bound does not require the mixing assumption
	\eqref{eq:mixing_env}; ergodicity of
	$(\Omega,\cF,\bP,\sigma)$ suffices.
	Whereas the decay rate in \eqref{eq:fcb_lsv} is governed by
	the fastest mixing map in the family, the rate in
	\eqref{eq:fcb_slow_lsv} is governed by the slowest mixing map.
	The latter bound, however, has a deterministic constant $C_*$.
	In particular, if $\overline{\beta}<1/4$, by applying Theorem \ref{thm:w2_sums}
	we obtain
	\begin{align*}
		\cW_2\bigl(\cL(W_{\omega,N}),N(0,1)\bigr)
		= O\bigl(N V_{\omega,N}^{-3/2}   + N^{1/2} V_{\omega, N}^{-1}  \bigr).
	\end{align*}
	Consequently, if $V_{\omega,N}^{-1}=O(N^{-1})$, we obtain
	the optimal rate of convergence in the central limit theorem
	in the $\cW_2$ distance:
	\begin{align}\label{eq:w2_lsv_optimal}
		\cW_2\bigl(\cL(W_{\omega,N}),N(0,1)\bigr)
		= O(N^{-1/2}).
	\end{align}
    To the best of our knowledge, this is the first optimal
$\cW_2$ rate in the central limit theorem for nonautonomous
intermittent maps. In the autonomous case, \eqref{eq:w2_lsv_optimal}
was established in \cite[Section 3.2]{DMR23} under the assumption that 
the parameter $\eta$ of the LSV map $T_\eta$ satisfies $\eta < 1/4$.
\end{remark}

\subsection{Application to random  piecewise non-uniformly  expanding maps}\label{Sec3.4}

Let $(\Omega, \mathcal F, \mathbb P)$ be a probability space equipped with measurable invertible preservation $\sigma \colon \Omega \to \Omega.$ We suppose that $\mathbb P$ is ergodic.
For each $\omega \in \Omega$, let $T_\omega$ be a piecewise $C^2$ map from $[0, 1]$ to itself and denote by $\lambda_\omega$ its minimal expansion of $T_\omega$, i.e. 
$\lambda_\omega:=\essinf |T_\omega'|$. As usual,
\[
T_\omega^n:=T_{\sigma^{n-1}\omega}\circ \ldots \circ T_{\sigma \omega}\circ T_\omega.
\]
Henceforth we will work with the class of maps and environments $(\Omega, \mathcal F, \mathbb P)$  in \cite[Section 5]{DH CMP}. These maps satisfy the following conditions.
\begin{enumerate}
\item[(i)] there exists of an equivariant family $\{\mu_\omega\}_{\omega \in \Omega}$ of Borel
probability measures on $[0, 1]$ such that $T_\omega^*\mu_\omega=\mu_{\sigma \omega}$ for $\mathbb P$-a.e. $\omega \in \Omega$. Moreover, $d\mu_\omega=h_\omega \, dm$, where $m=\text{Lebesgue}$. Furthermore,  there is $\tilde E\in L^{q_3}(\Omega, \mathbb P)$ 
for some $q_3 \ge 1$
such that 
\[
\|h_\omega\|_{\Lip}\le \tilde E(\omega), 
\]
for $\mathbb P$-a.e. $\omega \in \Omega$.
\vskip0.1cm
    \item[(ii)] there is a  random variable $E\colon \Omega \to (0, \infty)$, $E\in L^{q_0}(\Omega, \mathbb P)$ with $q_0 \ge 1$ and $a_0>1$ such that 
\begin{equation}\label{C2}
\prod_{i=0}^{n-1}\lambda_{\sigma^i \omega}^{-1}\le E(\omega)n^{-a_0},    
\end{equation}
for $\mathbb P$-a.e. $\omega \in \Omega$ and $n\in \mathbb N$. 
\vskip0.1cm
\item [(iii)] we have 
\begin{equation}\label{C3}
\|\mathcal L_\omega^n \phi \|_{L^1(m)}\le E'(\omega)n^{-a_1}\|\phi\|_{\Lip}    
\end{equation}
for $\mathbb P$-a.e. $\omega \in \Omega$, $n\in \mathbb N$ and $\phi$ Lipschitz with $\int \phi \, dm=0$. Here, $a_1>0$ and $E'\in L^{q_1}(\Omega, \mathbb P)$.
\vskip0.1cm
\item[(iv)] $\text{Lip}(\ln|T_{\omega}'|),\|T_\omega''\|_\infty\in L^{q}$ for some $q\ge1$.
\vskip0.1cm
\item[(v)] let $\mathcal A_{\omega}$ be the partition of $[0, 1]$ into the monotonicity intervals for $T_\omega$ and $|\mathcal A_\omega|$ denotes its cardinality. Then  $|\mathcal A_{\omega}|\in L^{r}$ for some $r\geq1$.
\vskip0.1cm
\item[(vi)] for every $p \in \bN$
we have
$$
\|\mathcal L_{\omega}^p\textbf{1}\|_{\infty}\leq Q(\omega)R(\sigma^p\omega)
$$
where $Q\in L^{t}$ and $R\in L^{z}$ for some $t, z \ge 1$.
\end{enumerate}

Before we proceed we need the following two consequences of our assumptions.
We first prove the following result.
Denote by $D(T)$ the smallest constant $D$ such that for all $I\in\mathcal A_T$ and all $x,y\in I$ we have $|T'(x)-T'(y)|\leq D|T'(x)||Tx-Ty|$.
We need the following result.
\begin{lemma}
Suppose that $\Lip(\ln|T_{\omega}'|)\in L^{q}$ for some $q>1$.
For every $\delta>0$ we have
$$
D(T_\omega^n)\leq E_0(\omega)n^{1/q_0+1/q+\delta}
$$
where $E_0(\omega)\in L^{s}(\Omega)$ with $1/s=1/q+1/q_0$.
\end{lemma}
\begin{proof}
For any $x,y \in I$, $I \in \cA_{ T_\omega^n }$, we have
\[
\left|\ln\frac{\left|(T_\omega^n)'(x)\right|}{\left|(T_\omega^n)'(y)\right|}\right|\leq \sum_{j=0}^{n-1}\left|\ln \left|T_{\sigma^j\omega}'(T_\omega^{j-1}x)\right|-\ln \left|T_{\sigma^j\omega}'(T_\omega^{j-1}y)\right|\right|
 \]
 \[
\leq \sum_{j=0}^{n-1}\text{Lip}(\ln|T_{\sigma^j\omega}'|)\left|T_\omega^{j-1}x-T_\omega^{j-1}y\right|.
 \]
 Next, we have 
 $$
\text{Lip}(\ln|T_{\sigma^j\omega}'|)\leq R_\delta(\omega)n^{1/q+\delta}
 $$
 with $R_\delta\in L^q$.
Since the map $T_{\sigma^k\omega}$ stretches distances by at least $\lambda_{\sigma^k\omega}$ and $x$ and $y$
belong to the same partition element, 
$$
\left|T_\omega^{j-1}x-T_\omega^{j-1}y\right|\leq \prod_{k=j}^{n-1}\lambda_{\sigma^k\omega}^{-1}|T_\omega^nx-T_\omega^n y|\leq E(\sigma^{j}\omega)(n-j)^{-a_0}|T_\omega^nx-T_\omega^n y|.
$$
Now, we have $E(\sigma^{j}\omega)\leq R'_\delta(\omega)n^{1/q_0+\delta}$ with $R'_\delta\in L^{q_0}$. The result follows now by the above estimates.
\end{proof}

The next result we need is.
\begin{lemma}
Suppose that $|\mathcal A_{\omega}|\in L^{r}$ and 
$$
\|\mathcal L_{\omega}^p\textbf{1}\|_{\infty}\leq Q(\omega)R(\sigma^p\omega)
$$
where $Q\in L^{t}$ and $R\in L^{z}$. Then for every $\delta>0$ we have 
\begin{equation}\label{der}
    \sum_{I\in\mathcal A_{\omega,n}}\sup_I\frac{1}{|(T_\omega^n)'|}\leq 2n^{1/q_2+\delta}E_{2,\delta}(\omega)
\end{equation}   
 where  $1/s=1/q+1/q_0$,
 $1/q_2=1/s+1/s'$,
 $1/s'=1/t+1/z$ and 
 $E_{2,\delta}\in L^{q_2}(\Omega,\mathbb P)$.
\end{lemma}
\begin{proof}
Notice that for all $I\in\mathcal A_{\omega,n}$ we have 
$$
\sup_I\frac1{|(T_\omega^n)'|}\leq \inf_{I}\frac1{|(T_\omega^n)'|}\left(1+D(T_\omega^n)\|(T_{\sigma^{n-1}\omega})''\|_{\infty}\right).
$$
Now, let $\delta>0$. Since $\|(T_\omega)''\|_\infty\in L^q$ we have $\|(T_{\sigma^{n-1}\omega})''\|_{\infty}\leq n^{1/q+\delta}U(\omega)$ with $U\in L^q$. Similarly,  
$R(\sigma^n\omega)\leq U_1(\omega)n^{1/z+\delta}$ with $U_1\in L^z$.
Using these estimates we conclude that 
$$
    \sum_{I\in\mathcal A_{\omega,n}}\sup_I\frac{1}{|(T_\omega^n)'|}\leq \mathcal \|\mathcal L_{\omega}^n\textbf{1}\|_\infty\left(1+ E_0(\omega)n^{1/q_0+1/q+\delta}U(\omega)n^{1/q+\delta}\right)
$$
$$
\leq Q(\omega)U_1(\omega)n^{1/z+\delta}\left(1+ E_0(\omega)n^{1/q_0+1/q+\delta}U(\omega)n^{1/q+\delta}\right).
$$
\end{proof}

Next, let $F\colon [0, 1]^k \to \mathbb R$ be a separately Lipschitz 
continuous function. In what follows we will estimate
\[
J=J(\omega):=\int_{0}^1H_\omega(x, x)\, d\mu_\omega(x)-\int_0^1 \int_0^1H_\omega(x, y)\, d\mu_\omega(x)\, d\mu_{\omega}(y),
\]
where $H\colon [0, 1]^2\to \mathbb R$ is given by 
\[
H_\omega(x, y):=F(T_{\omega}^{n_1}(x), \ldots, T_{\omega}^{n_l}(x), T_{\omega}^{n_{l+1}}(y), \ldots, T_\omega^{n_k}(y)).
\]
Here, $0\le n_1<n_2<\ldots <n_k$ is a fixed strictly increasing finite sequence of positive integers independent of $\omega$, $x$ and $y$.
\begin{proposition}\label{J prop}
For every $\delta>0$ there are $E_\delta \in L^{q_0}(\Omega, \mathbb P)$, $E_\delta'\in L^{q_1}(\Omega, \mathbb P)$, $\tilde E_\delta\in L^{q_3}(\Omega,\mathbb P)$, $E_\delta'''\in L^{b}(\Omega,\mathbb P)$ with $1/b=2/q_0+1/q$ and $\tilde c\in L^{b}(\Omega,\mathbb P)$
such that 
\[
\begin{split}
J &=O(E_\delta(\omega)n_l^{\frac 1 q+\delta}(n_{l+1}-n_l)^{1-a_0})+O(E_\delta'(\omega)\tilde E_\delta(\omega)n_{l+1}^{\frac{1}{q_1}+\frac{1}{q_3}+2\delta}(n_{l+1}-n_l)^{-a_1})\\
&\phantom{=}+O(E_\delta'(\omega)\tilde E(\omega)E''(\omega)E_\delta'''(\omega)(1+\tilde c_\delta(\omega)) n_{l+1}^{\frac{1}{q_1}+\frac{1}{q_2}+\frac{1}b+3\delta}(n_{l+1}-n_l)^{-a_1}).
\end{split}
\]
Here the $O(\cdot)$ terms are bounded by some absolute constant $C$ times $\|F\|_{\Lip}$ and the parameters in the brackets.
\end{proposition}
Let $f:\Omega\times [0,1]\to\mathbb R$ be measurable function such that a   $\|f(\omega,\cdot)\|_{\Lip}\in L^p$ for some $2<p<\infty$. Let $S_n^\omega f(s)=\sum_{j=0}^{n-1}f(\sigma^j\omega, T_\omega ^j(x))$ considered as a random variable on $([0,1],\mu_\omega)$. Suppose that $\tilde f(\omega,x)=f(\omega,x)-\mu_\omega(f(\omega,\cdot))$ is not a coboundary with respect to the skew product $\tau(\omega,x)=(\sigma\omega, T_\omega x)$. Then (see \cite{DH CMP}) there exists $\Sigma>0$ such that for $\mathbb P$-a.a. $\omega$ we have
$$
\lim_{n\to\infty}\frac1n\text{Var}(S_n^\omega f)\to\infty=\Sigma^2.
$$
Let 
$$
W_n^\omega=\frac{S_{n}^\omega f-\mu_\omega(S_n^\omega f)}{\sigma_{\omega,n}}
$$
where $\sigma_{\omega,n}=\sqrt{\text{Var}_{\mu_\omega}(S_n^\omega f)}$.

The following result is a consequence of Proposition \ref{J prop} together with Theorem \ref{thm:w2_sums} and the mean ergodic theorem.
\begin{corollary}
Suppose that $a_0>4$ and $a_1>3$. 
For $\mathbb P$-a.a. $\omega$ for every $\delta>0$ we have
$$
\cW_2(W_N^\omega,N(0,1))=O(N^{-1/2+c+\delta})
$$
where
$$
c=4/p+1/q_0+2/q_1+2/q_3+3/b.
$$
\end{corollary}

\subsubsection{Estimate of $J$: proof of Proposition \ref{J prop}}

Set
\[
n_*:=n_l+\lfloor \frac{n_{l+1}-n_l}{2}\rfloor.
\]
By $\mathcal A_{\omega, n_*}$ we denote the partition of $[0, 1]$ into the monotonicity intervals for $T_\omega^{n_*}$.
We apply Proposition \ref{prop:scheme} with
$$
\cQ =   \cQ^{ (g) } = \mathcal A_{\omega, n_*} \quad \text{and} \quad \cQ^{ (b) } = \emptyset.
$$
It remains to verify (A2) and (A3), (A1) being automatically true.
Throughout the proof, we restrict to $\omega$ in a $\sigma$-invariant
subset of $\Omega$ with full $\mathbb P$-measure on which
assumptions (i)-(vi) hold.

\medskip

\noindent\textbf{(A2):} For $I \in  \mathcal A_{\omega, n_*}$,
\begin{align*}
	\sum_{j=1}^l \text{diam}( T_\omega^{n_j}( I  ) ) \le \sum_{i=1}^l \prod_{j=n_i+1}^{n_*}\lambda_{\sigma^j \omega}^{-1}.
\end{align*}
Next, by \eqref{C2}, 
\[
\prod_{j=n_i+1}^{n_*}\lambda_{\sigma^j \omega}^{-1}\le E(\sigma^{n_i+1}\omega)(n_*-n_i)^{-a_0}\le E_\delta(\omega)n_i^{\frac 1 q+\delta}(n_*-n_i)^{-a_0},
\]
where $\delta>0$ is arbitrary and $E_\delta \in L^{q_0}(\Omega, \mathbb P)$. Thus, for $1\le i \le n$ we have
\[
\prod_{j=n_i+1}^{n_*}\lambda_{\sigma^j \omega}^{-1} \le E_\delta(\omega)n_l^{\frac 1 q+\delta}(n_*-n_i)^{-a_0},
\]
implying that, for some constants $D, D_1 > 0$,
\[
	\sum_{j=1}^l \text{diam}( T_\omega^{n_j}( I  ) ) \le DE_\delta(\omega)n_l^{\frac 1 q+\delta}(n_*-n_l)^{1-a_0} \le 
	D_1 E_\delta(\omega)n_l^{\frac 1 q+\delta}(n_{l+1}-n_l)^{1-a_0}.
\]

\noindent\textbf{(A3):} Fix $c_I \in I$ for each $I \in \cA_{\omega, n_*}$. Then,
\[
\begin{split}
&\int_0^1 \left \{ \left (\tilde H(c_I, \cdot)-\int_0^1 H(c_I, y)\, d\mu_\omega(y)\right )\circ T_{\omega}^{n_{l+1}}
\right \}\cdot (\mathbf 1_I(\cdot)-\mu_\omega(I))\, d\mu_\omega \\
&=\int_0^1 \left (\tilde H(c_I, \cdot)-\int_0^1 H(c_I, y)\, d\mu_\omega(y)\right ) \cdot L_\omega^{n_{l+1}}(\mathbf 1_I-\mu_\omega(I))\, d\mu_{\sigma^{n_{l+1}}\omega}.
\end{split}
\]
Here, $L_\omega$ are normalized transfer operators, and $\tilde H(x, T_{\omega}^{n_{l+1}}(y))=H(x, y)$.  Denoting by $\mathcal L_\omega$ classical transfer operators, we have 
\[
\begin{split}
&\int_I \left [H(c_I, x)-\int_0^1H(c_I, y)\, d\mu_\omega(y)\right ]\, d\mu_\omega(x) \\
&=\int_0^1 \left (\tilde H(c_I, \cdot)-\int_0^1 H(c_I, y)\, d\mu_\omega(y)\right ) \cdot \mathcal L_\omega^{n_{l+1}}(h_\omega \mathbf 1_I-\mu_\omega(I) h_\omega )\, dm.
\end{split}
\]
Thus
\begin{align}
\begin{split}
\label{eq:aux1}
&\left|
\int_0^1 \left \{ \left (\tilde H(c_I, \cdot)-\int_0^1 H(c_I, y)\, d\mu_\omega(y)\right )\circ T_{\omega}^{n_{l+1}}
\right \}\cdot (\mathbf 1_I(\cdot)-\mu_\omega(I))\, d\mu_\omega \right| \\
&\le 2\|F\|_\infty \cdot \|\mathcal L_{\sigma^{n_*}\omega}^{n_{l+1}-n_*}\mathcal L_\omega^{n_*}(h_\omega \mathbf 1_I-\mu_\omega(I) h_\omega)\|_{L^1(m)}.
\end{split}
\end{align}
Next, by \eqref{C3}
\begin{align}\label{eq:aux2}
\begin{split}
&\|\mathcal L_{\sigma^{n_*}\omega}^{n_{l+1}-n_*}\mathcal L_\omega^{n_*}(h_\omega \mathbf 1_I-\mu_\omega(I) h_\omega)\|_{L^1(m)} \\
&\le E'(\sigma^{n_*}\omega)(n_{l+1}-n_*)^{-a_1}\|\mathcal L_\omega^{n_*}(h_\omega \mathbf 1_I-\mu_\omega(I) h_\omega)\|_{\Lip}\\
&\le E'_\delta(\omega)n_*^{\frac{1}{q_1}+\delta}(n_{l+1}-n_*)^{-a_1}\|\mathcal L_\omega^{n_*}(h_\omega \mathbf 1_I-\mu_\omega(I) h_\omega)\|_{\Lip},
\end{split}
\end{align}
where $\delta>0$ is arbitrary and $E_\delta'\in L^{q_1}(\Omega, \mathbb P)$.

It remains to estimate
\[
\|\mathcal L_\omega^{n_*}(h_\omega \mathbf 1_I-\mu_\omega(I) h_\omega)\|_{\Lip}.
\]
We have
\begin{align}\label{eq:aux3}
\|\mathcal L_\omega^{n_*}(h_\omega \mathbf 1_I-\mu_\omega(I) h_\omega)\|_{\Lip}\le \|\mathcal L_\omega^{n_*}(h_\omega \mathbf 1_I)\|_{\Lip}+\mu_\omega(I)\|h_{\sigma^{n_*}\omega}\|_{\Lip}.
\end{align}
Moreover, 
\[
\|\mathcal L_\omega^{n_*}(h_\omega \mathbf 1_I)\|_{\Lip}=\|\mathcal L_\omega^{n_*}(h_\omega \mathbf 1_I)\|_\infty+\Lip(\mathcal L_\omega^{n_*}(h_\omega \mathbf 1_I)),
\]
where $\Lip(\mathcal L_\omega^{n_*}(h_\omega \mathbf 1_I))$ denotes the Lipschitz constant of $\mathcal L_\omega^{n_*}(h_\omega \mathbf 1_I)$. We have 
\[
\mathcal L_\omega^{n_*}(h_\omega \mathbf 1_I)(x)=\frac{h_\omega((T_{\omega}^{n_*}\rvert_I)^{-1}(x))}{(T_\omega^{n_*})'((T_{\omega}^{n_*}\rvert_I)^{-1}(x))}.
\]
Therefore,
\begin{align}\label{eq:aux4}
\|\mathcal L_\omega^{n_*}(h_\omega \mathbf 1_I)\|_\infty \le \|h_\omega\|_\infty \cdot \sup_I \left |\frac{1}{(T_\omega^{n_*})'}\right |.
\end{align}

Moreover, by \cite[Lemma 17]{DH CMP} for  $x, x'\in [0, 1]$ we have
\[
\begin{split}
&|\mathcal L_\omega^{n_*}(h_\omega \mathbf 1_I)(x)-\mathcal L_\omega^{n_*}(h_\omega \mathbf 1_I)(x')|\\
&\le 2\|h_\omega\|_{Lip}(1+c_1(\sigma^{n_*}\omega))\|1/(T_{\omega}^{n_*})'\|_{\infty,I}|x-x'|
\end{split}
\]
where $\|g\|_{\infty,I}=\sup_{I}|g|$ and $c_1\in L^{b_0}$. Note that $c_1(\sigma^n\omega)\leq \tilde c_\delta(\omega)n^{1/b+\delta}, \delta>0$ with $\tilde c_\delta\in L^b$. 
Here we used that 
$$
\mathcal L_\omega^{n_*}(h_\omega \mathbf 1_I)(x)=h_\omega(y_I(x))y_I'(x)
$$
where $y_I(x)=(T_{\omega}^{n_*}\rvert_I)^{-1}(x)$ is the inverse branch of $T_{\omega}^{n^*}$ corresponding to $I$. Therefore,
\begin{align}\label{eq:aux5}
	\Lip (\mathcal L_\omega^{n_*}(h_\omega \mathbf 1_I))
	\le 2\|h_\omega\|_{Lip}(1+c_1(\sigma^{n_*}\omega))\|1/(T_{\omega}^{n_*})'\|_{\infty,I}.
\end{align}

Let
$$
\cI = \sum_{I\in \mathcal A_{\omega, n_*}}  \int_0^1 \left \{ \left (\tilde H(c_I, \cdot)-\int_0^1 H(c_I, y)\, d\mu_\omega(y)\right )\circ T_{\omega}^{n_{l+1}}
\right \}\cdot (\mathbf 1_I(\cdot)-\mu_\omega(I))\, d\mu_\omega.
$$
By \eqref{eq:aux1}, \eqref{eq:aux2}, and \eqref{eq:aux3},
\begin{align*}
	|\cI| 
	&\le 
	 C\sum_{I\in \mathcal A_{\omega, n_*}}\|\mathcal L_{\sigma^{n_*}\omega}^{n_{l+1}-n_*}\mathcal L_\omega^{n_*}(h_\omega \mathbf 1_I-\mu_\omega(I) h_\omega)\|_{L^1(m)}\\
	&\le CE'_\delta(\omega)n_*^{\frac{1}{q_1}+\delta}(n_{l+1}-n_*)^{-a_1}\sum_{I\in \mathcal A_{\omega, n_*}}\|\mathcal L_\omega^{n_*}(h_\omega \mathbf 1_I-\mu_\omega(I) h_\omega)\|_{Lip}\\
	&\le CE'_\delta(\omega)n_*^{\frac{1}{q_1}+\delta}(n_{l+1}-n_*)^{-a_1}\sum_{I\in \mathcal A_{\omega, n_*}}\left(\|\mathcal L_\omega^{n_*}(h_\omega \mathbf 1_I)\|_{Lip}+\mu_\omega(I)\|h_{\sigma^{n_*}\omega}\|_{\Lip}\right ).
\end{align*}
Recall from our assumption (i)  that there is $\tilde E\in L^{q_3}(\Omega, \mathbb P)$ such that 
\[
\|h_\omega\|_{\Lip}\le \tilde E(\omega).
\]
Using this property together with \eqref{eq:aux4} and \eqref{eq:aux5}, 
it follows that
\begin{align*}
	|\cI| &\le 
	CE_\delta'(\omega)\tilde E_\delta(\omega)n_*^{\frac{1}{q_1}+\frac{1}{q_3}+2\delta}(n_{l+1}-n_*)^{-a_1}\\
	&\phantom{\le}+CE'_\delta(\omega)n_*^{\frac{1}{q_1}+\delta}(n_{l+1}-n_*)^{-a_1}\sum_{I\in \mathcal A_{\omega, n_*}}\|\mathcal L_\omega^{n_*}(h_\omega \mathbf 1_I)\|_{\Lip} \\
	&\le CE_\delta'(\omega)\tilde E_\delta(\omega)n_*^{\frac{1}{q_1}+\frac{1}{q_3}+2\delta}(n_{l+1}-n_*)^{-a_1}\\
	&\phantom{\le}+CE'_\delta(\omega)\tilde E(\omega) n_*^{\frac{1}{q_1}+\delta}(n_{l+1}-n_*)^{-a_1}\sum_{I\in \mathcal A_{\omega, n_*}}\sup_I \left |\frac{1}{(T_\omega^{n_*})'}\right| \\
	&\phantom{\le}+CE'_\delta(\omega)n_*^{\frac{1}{q_1}+\delta}(n_{l+1}-n_*)^{-a_1}\sum_{I\in \mathcal A_{\omega, n_*}} \Lip (\mathcal L_\omega^{n_*}(h_\omega \mathbf 1_I)) \\
	&\le 
	CE_\delta'(\omega)\tilde E_\delta(\omega)n_*^{\frac{1}{q_1}+\frac{1}{q_3}+2\delta}(n_{l+1}-n_*)^{-a_1}\\
	&\phantom{\le}+CE'_\delta(\omega)\tilde E(\omega) n_*^{\frac{1}{q_1}+\delta}(n_{l+1}-n_*)^{-a_1}\sum_{I\in \mathcal A_{\omega, n_*}} \|1/(T_{\omega}^{n_*})'\|_{\infty,I} \\
	&\phantom{\le}+CE'_\delta(\omega)\tilde E(\omega) n_*^{\frac{1}{q_1}+1/b+2\delta}(n_{l+1}-n_*)^{-a_1}\sum_{I\in \mathcal A_{\omega, n_*}}(1+\tilde c_\delta(\omega))\|1/(T_{\omega}^{n_*})'\|_{\infty,I}.
\end{align*}
Here $\|g\|_{\infty,I}=\sup_{x\in I}|g(x)|$.
Finally, using \eqref{der} we obtain
\begin{align*}
	|\cI| &\le CE_\delta'(\omega)\tilde E_\delta(\omega)n_*^{\frac{1}{q_1}+\frac{1}{q_3}+2\delta}(n_{l+1}-n_*)^{-a_1}\\
	&\phantom{\le}+CE'_\delta(\omega)\tilde E(\omega)(1+\tilde c_\delta(\omega)) n_*^{\frac{1}{q_1}+\frac{1}{b}+2\delta}(n_{l+1}-n_*)^{-a_1}\sum_{I\in \mathcal A_{\omega, n_*}}\|1/(T_{\omega}^{n_*})'\|_{\infty,I}\\
	&\le CE_\delta'(\omega)\tilde E_\delta(\omega)n_{l+1}^{\frac{1}{q_1}+\frac{1}{q_3}+2\delta}(n_{l+1}-n_l)^{-a_1}\\
	&\phantom{=}+CE_\delta'(\omega)\tilde E(\omega)E''(\omega)E_\delta'''(\omega)(1+\tilde c_\delta(\omega)) n_{l+1}^{\frac{1}{q_1}+\frac{1}{q_2}+\frac{1}{b}+3\delta}(n_{l+1}-n_l)^{-a_1}.
\end{align*}

It follows by Proposition \ref{prop:scheme} that
\[
\begin{split}
J &=O(E_\delta(\omega)n_l^{\frac 1 q+\delta}(n_{l+1}-n_l)^{1-a_0})+O(E_\delta'(\omega)\tilde E_\delta(\omega)n_{l+1}^{\frac{1}{q_1}+\frac{1}{q_3}+2\delta}(n_{l+1}-n_l)^{-a_1})\\
&\phantom{=}+O(E_\delta'(\omega)\tilde E(\omega)E''(\omega)E_\delta'''(\omega)(1+\tilde c_\delta(\omega)) n_{l+1}^{\frac{1}{q_1}+\frac{1}{q_2}+\frac{1}b+3\delta}(n_{l+1}-n_l)^{-a_1}).
\end{split}
\]
\qed

\section{Optimal CLT $\cW_2$-rates for uniformly bounded $\phi$-mixing sequences}\label{Sec4}
Let us consider  real-value zero mean uniformly bounded random variables $X_0,X_1,...$ on some probability space $(\Omega,\mathcal F,\mathbb P)$. Recall that the $\phi$-mixing (weak dependence) coefficients of the sequence are given by 
$$
\phi(n)=\sup_k\sup\left\{|\mathbb P(B|A)-\mathbb P(B)|: \mathbb P(A)>0, A\in\sigma\{X_i: i\leq k\}, B\in\sigma\{X_i: i\geq k+n\}\right\}.
$$
Let $S_N=\sum_{n=0}^{N-1}X_n$ and let $V_N=\text{Var}(S_N)$.

In this section we prove the following result.
\begin{theorem}\label{PhiMix}
Suppose that 
$$
\sum_{j=1}^\infty j^{2}\phi(j)<\infty.
$$
Then with $W_N=(V_N)^{-1/2}S_N$ there
  is an absolute constant $C>0$ such that with $\gamma=\sup_j\|X_j\|_{L^\infty}$,
$$
\cW_2(W_N,N(0,1))\leq C\gamma^3(NV_N^{-3/2}+\sqrt{N}V_N^{-1}).
$$
In particular, if $V_N\geq cN$ for some $c>0$ and all $N$ large enough then
$$
\cW_2(W_N,N(0,1))=O(N^{-1/2}).
$$
\end{theorem}
\begin{remark}\label{RemPhi}
In the weakly stationary case when $X_n$ is not a coboundary and $\sum_n n\phi(n)<\infty$ there exists $\sigma>0$ such that 
$$
\sup_N\left|V_N-\sigma^2 N\right|<\infty.
$$
Therefore,
$$
\left\|W_N-\frac{S_N}{\sigma\sqrt{N}}\right\|_{L^2}=O(1/N).
$$
Using also the Berkes-Phillip lemma \cite[Lemma A.1]{BP} we conclude that when $\sum_n n^2\phi(n)<\infty$ then with $\widetilde W_N=\frac{S_N}{\sigma\sqrt{N}}$ we have
$$
\cW_2(\widetilde W_N,N(0,1))=O(N^{-1/2}).
$$
\end{remark}

\begin{remark}
A notable result of Rio \cite{Rio} shows for stationary sequences with $\sum_{j}j\phi(j)<\infty$ and if $V_N\to\infty$ (and so $V_N/N\to s$ for some $s>0$) the sequences $S_N/\sqrt{sN}$ and $S_N/\sqrt{V_N}$ have optimal uniform CLT rates $O(N^{-1/2})$. Recently \cite{Aust} obtained optimal CLT rates in $\mathcal W_p$ for polynomially fast stationary \textbf{interlaced }$\alpha$-mixing sequences. Also in \cite{DMR23} optimal $\cW_2$-rates where obtained for stationary $\alpha$-mixing sequences with values on a Polish space.
To the best of our knowledge these are the only results in literature when optimal rates  are obtained for non-Markovian sequences or sequences that are not formed by a sufficiently chaotic dynamical system. Note that in Theorem \ref{Psi} we establish the same for unbounded variables, under the stronger $\psi$-mixing condition. 
\end{remark}

\begin{corollary}
If $(X_n)_{n\geq0}$ is weakly stationary\footnote{Namely, $\mathbb E[X_{n+k}X_k]$ depends only on $n$.}, $\sum_j j^{2}\phi(j)<\infty$ 
and it does not have the form $X_n=Y_{n+1}-Y_n$ for some other weakly stationary square integrable sequence $Y_n$ then
$$
\cW_2(W_N,N(0,1))\leq C\gamma^3 N^{-1/2}
$$
and 
$$
\cW_{2}(\widetilde W_N, N(0,1))\leq  C\gamma^3 N^{-1/2}
$$
for some constant $C>0$.
\end{corollary}
\begin{proof}
In this case the limit 
$$
s=\lim_{N\to\infty}\frac1N\text{Var}(S_N)
$$
exists and it is positive.
\end{proof}

Before proving Theorem \ref{PhiMix}, we recall the following result (see \cite[Corollary 3.3]{YeorStien} or \cite[Corollary 1.3.11]{HK}).
\begin{proposition}\label{YeorStien}
Let $I_1,I_2,...,I_\ell$ be intervals in the positive integers ordered so that $I_i$ is to the left of $I_{i+1}$. Denote by $d_i$ the gap between $I_{i}$ and $I_{i-1}$, where $i\geq 2$. 
Let $Y_j$ be random vectors  measurable with respect to $\sigma\{X_k: k\in I_j\}$. Let $\{\mathcal C_1,...,\mathcal C_d\}$ be a partition of $\{1,...,\ell\}$. Define $Y_{\cC}=(Y_j)_{j\in\mathcal C}$ and let $\widetilde Y_{\cC_j}$ be independent copies of $Y_{\cC_j}$. Let $H=H(y_1,...,y_\ell)$ be a measurable function and suppose that $\|H(Z_1,...,Z_\ell)\|_{L^\infty}<\infty$ where $Z_j$ is either $Y_j$ or $\widetilde  Y_j$, where $\widetilde  Y_j, 1\leq j\leq \ell$ are independent copies of $Y_j$. Let $\gamma$ be the maximum of these $2^\ell$ essential suprema.
Let us view $H$ as a function 
$H=H(y_{\cC_1},...,y_{\cC_d})$ where $y_{\cC}=(y_j)_{j\in\cC}$. Then
$$
\left|\mathbb E[H(Y_{\cC_1},...,Y_{\cC_d})]-\mathbb E[H(\widetilde  Y_{\cC_1},...,\widetilde  Y_{\cC_d})]\right|\leq 4\gamma\sum_{k=2}^{d}\phi(d_k).
$$    
\end{proposition}

\begin{proof}[Proof of Theorem \ref{PhiMix}]
 Notice that it is enough to prove the theorem when $\gamma\leq 1$, for otherwise we can just normalize by $\gamma$ (the factor $\gamma^3$ from scaling the variance).
Take $d(n,m)=|n-m|$. Let us estimate all the error term in Proposition \ref{prop:general_w2}.

Throughout the proof, we use $C$ to denote a generic absolute constant.

\noindent\textbf{Bound on $|\zeta|$.} 
By applying Proposition \ref{YeorStien} we see that 
$$
|\zeta|\leq CV_N^{-3/2}\sum_{n=0}^{N-1}\sum_{m=1}^{N-1}\sum_{k=3}^{N-1}\phi(m(k-2)+1)+
CV_N^{-3/2}\sum_{n=0}^{N-1}\sum_{m=1}^{N-1}\phi(m)
$$
$$
+CV_N^{-3/2}\sum_{n=0}^{N-1}\sum_{k=1}^{N-1}\phi(k)+CV_N^{-3/2}\sum_{n=0}^{N-1}\sum_{m=1}^{N-1}\phi(m)+CV_N^{-3/2}N\leq C_0NV_N^{-3/2}.
$$

\noindent\textbf{Bound on $|\mathcal E_1|$.}
By   Proposition \ref{YeorStien}, and using that $f''$ is Lipschitz continuous,  for each $u\in[0,1]$ we have
$$
\left|E \biggl\{ \sum_{n=0}^{N-1} 
			\sum_{m=0}^{N-1}   Y_n (Y_{n,m})^2 
			  \delta_f(   W_{n,m} + uY_{n,m}, W_{n,m}) 
			\biggr\}\right|
$$
$$
 \leq CV_N^{-2}N\sum_{m=1}^{N-1}\phi(m)\leq CV_N^{-2}N
$$
and so 
$$
|\mathcal E_1|\leq C_0V_N^{-2}N.
$$

\noindent\textbf{Bound on $|\mathcal E_2|$.}
Let us write 
$$
\mathcal E_2=I_1+I_2+I_3
$$
where $I_i$ is the $i$-th sum. By Proposition \ref{YeorStien} we have 
$$
|I_1|\leq CNV_N^{-2}+CV_N^{-2}\sum_{n=0}^{N-1}\sum_{m=1}^{N-1}m^{2}\phi(m)\leq 
CNV_N^{-2}\leq 
C'V_N^{-2}N.
$$
Moreover, using the Lipschitz continuity of $f''$ we see that
$$
|I_2|\leq CNV_N^{-2}+CV_N^{-2}\sum_{n=0}^{N-1}\sum_{k=2}^{N-1}\phi(k-2)\leq C'NV_N^{-2}.
$$
Furthermore, 
$$
|I_3|\leq CV_N^{-2}\sum_{n=0}^{N-1}\sum_{m=1}^{N-1}\sum_{j\in[n]_m}\sum_{k=2}^{N-1}m\phi(m)\phi(m(k-2))
$$
$$
\leq CV_{N}^{-2}N\sum_{m=1}^{N-1}m^2\phi(m)\sum_{k=2}^\infty\phi(m(k-2))\leq C'V_N^{-2}N.
$$
We thus conclude that 
$$
|\mathcal E_2|\leq C''V_N^{-2}N.
$$

\noindent\textbf{Bound on $|\mathcal E_3|$.}
Take $u\in[0,1]$. Then by  Proposition \ref{YeorStien},
$$
\left|E\biggl\{ \sum_{n=0}^{N-1} 
			\sum_{m=1}^{N-1} \sum_{k=2}^{N-1}
			\overline{
				\delta_f(  W_{n,mk}  + u Y_{n, \langle m(k-1), mk \rangle } ,  W_{n, 2mk} ) 
				Y_{n, \langle m(k-1), mk \rangle } }  
			Y_{n,m} Y_n \biggr\}\right|
$$
$$
\leq CV_N^{-2}\sum_{n=0}^{N-1} 
			\sum_{m=1}^{N-1}\left(m\phi(m)+\sum_{k=3}^{N-1}mk\phi(m(k-2))\right) 
$$
$$
\leq C'V_N^{-2}N+C'V_N^{-2}N\sum_{m=1}^\infty m^2\phi(m)\leq C''V_N^{-2}N.
$$
Moreover,
$$
\left|E\biggl\{ \sum_{n=0}^{N-1} \sum_{k=1}^{N-1} 
			\overline{ 
				\delta_f( W_{n,k} + u Y_{n,k},  W_{n, 2k} )    Y_{n,k} } 
			Y_n^2 \biggr\}\right|
$$
$$
\leq CV_N^{-2}N\sum_{k=1}^N k\phi(k)\leq C'V_N^{-2}N.
$$
Thus, 
$$
|\cE_3|\leq \tilde CV_N^{-2}N.
$$

The rest of the terms are bounded similarly. For instance, let us handle the most complicated expression
$$
E \biggl\{ \sum_{n=0}^{N-1} 
			\sum_{m=1}^{N-1} \sum_{k=2}^{N-1} 
			\sum_{p=3}^{N-1}  \overline{   
			\delta_f( W_{n, mk (p-1)}, W_{n, mkp}  )	
				 }
			Y_{n, \langle m(k-1), mk \rangle}  Y_{n,m} Y_n \biggl\}.
$$
By  Proposition \ref{YeorStien} and the Lipschitz continuity of $f''$ we have
$$
\left|E \biggl\{\overline{   
			\delta_f( W_{n, mk (p-1)}, W_{n, mkp}  )	
				 }
			Y_{n, \langle m(k-1), mk \rangle}  Y_{n,m} Y_n \biggl\}\right|\leq CV_N^{-2}m^2k\phi(mk(p-2)).
$$
Now,
$$
\sum_{n=0}^{N-1} 
			\sum_{m=1}^{N-1} \sum_{k=2}^{N-1} 
			\sum_{p=3}^{N-1}m^2k\phi(mk(p-2))\leq CN\sum_{m=1}^{N-1}\sum_{k=2}^{N-1}\sum_{\ell=1}^{N-3}m^2k\phi(mk\ell)\leq CN\sum_{s=1}^{\infty}s^2\phi(s).
$$
\end{proof}

\section{Optimal CLT $\cW_2$-rates for  $\psi$-mixing sequences under the fourth moment condition}\label{Sec5}
Let us consider a sequence of real-value zero mean square integrable random variables $X_0,X_1,...$ on some probability space $(\Omega,\mathcal F,\mathbb P)$. Recall that the $\psi$-mixing (weak dependence) coefficients of the sequence are given by 
$$
\psi(n)=\sup_k\sup\left\{\left|\frac{\mathbb P(A\cap B)}{\mathbb P(A)\mathbb P(B)}-1\right|: \mathbb P(A)\mathbb P(B)>0,\, A\in\sigma_{-\infty,k},\,
B\in\sigma_{k+n,\infty}\right\}
$$
where $\sigma_{-\infty,k}=\sigma\{X_i: i\leq k\}$ and $\sigma_{k+n,\infty}=\sigma\{X_i: i\geq k+n\}$.
Let $S_N=\sum_{n=0}^{N-1}X_n$ and let $V_N=\text{Var}(S_N)$. 

\begin{theorem}\label{Psi}
Suppose that $\gamma_4=\sup_j\|X_j\|_{L^4}<\infty$ and that 
$$
\sum_{j=1}^\infty j^2\psi(j)<\infty.
$$ 
Then   there is an absolute constant $C>0$ such that 
 with $W_N=(V_N)^{-1/2}S_N$,
$$
\cW_2(W_N,N(0,1))\leq C\gamma_4^3(NV_N^{-3/2}+\sqrt{N}V_N^{-1}).
$$
In particular, if $V_N\geq cN$ for some $c>0$ and all $N$ large enough then 
$$
\cW_2(W_N,N(0,1))=O(N^{-1/2}).
$$
Consequently, if $(X_j)$ is weakly stationary and it does not have the form $X_n=Y_{n+1}-Y_n$ then 
$$
\cW_2(W_N,N(0,1))=O(N^{-1/2}).
$$
\end{theorem}
\begin{remark}
As noted in the previous section, for sufficiently fast uniformly bounded stationary $\phi$-mixing sequences $X_j$ optimal uniform CLT rates are known. However, to the best of our knowledge, Theorem \ref{Psi} is the first result where optimal CLT rates in any sense are obtained under assumptions on regular (not interlaced) mixing coefficients for unbounded variables $X_j$. Moreover, this is obtained under the fourth moment condition $\gamma_4<\infty$ and the dependence on $\gamma_4$ is like in the best known results in the iid case.  
\end{remark}
\begin{remark}
 Like in Remark \ref{RemPhi}, we get optimal rates in the weakly stationary case with the $\sqrt{N}$-normalization.   
\end{remark}
\subsection*{Proof of Theorem \ref{Psi}.}
First, recall the following result \cite[Lemma 3.1]{Kif}. 
\begin{lemma}\label{KifLemm}
Let $\mathcal G,\mathcal H\subset\mathcal F$ be two sub-$\sigma$-algebras. Let 
$$
\psi(\cG,\cH)=
\sup\left\{\left|\frac{\mathbb P(A\cap B)}{\mathbb P(A)\mathbb P(B)}-1\right|: \mathbb P(A)\mathbb P(B)>0, A\in\cG, B\in\cH\right\}.
$$
Let $X$ be a random vector measurable with respect to $\cG$ an $Y$ be a random vector measurable with respect to $\cH$. Let $H(x,y)$ be a function such that for all $x$, $r(x)=\mathbb E[|H(x,Y)|]<\infty$ and $\mathbb E[|H(X,Y)|]<\infty$. Set $g(x)=\mathbb E[H(x,Y)]$. Then almost surely we have
$$
\left|\mathbb E[H(X,Y)|\mathcal G]-g(X)\right|\leq\psi(\cG,\cH)r(X).
$$
In particular, if $\widetilde X$ and $\widetilde Y$ are two independent copies of $X$ and $Y$ then 
$$
\left|\mathbb E[H(X,Y)]- \mathbb E[H(\widetilde  X,\widetilde Y)]\right|\leq \psi(\cG,\cH)\mathbb E[|H(\widetilde X,\widetilde  Y)|].
$$
 \end{lemma}
 
 The following result follows by induction.
\begin{corollary}
Let $I_1,I_2,...,I_\ell$ be intervals in the positive integers ordered so that $I_i$ is to the left of $I_{i+1}$. Denote by $d_i$ the gap between $I_{i}$ and $I_{i-1}$, where $i>2$. 
Let $Y_j$ be random vectors  measurable with respect to $\sigma\{X_k: k\in I_j\}$, and let $\widetilde Y_j$ be independent copies.   
Let $H=H(y_1,...,y_\ell)$ be a measurable function.
Suppose that for each $j$  we have $\gamma_j=\mathbb E[|H(Y_{1},...,Y_{j-1},\widetilde  Y_{j},...,\widetilde  Y_\ell)|]<\infty$. 
Let $\gamma=\max\{\gamma_j: 1\leq j\leq\ell+1\}$.
Then 
$$
\left|\mathbb E[H(Y_1,...,Y_\ell)]-\mathbb E[H(\widetilde Y_1,...,\widetilde Y_\ell)]\right|\leq\gamma\sum_{k=2}^{\ell}\psi(d_k).
$$
\end{corollary}

Using the above corollary and the arguments in the proof of \cite[Corollary 3.3]{YeorStien} or \cite[Corollary 1.3.11]{HK} we get the following result, which is a version of Proposition \ref{YeorStien}.
\begin{corollary}\label{Corrr}
Let $I_1,I_2,...,I_\ell$ be intervals in the positive integers ordered so that $I_i$ is to the left of $I_{i+1}$. Denote by $d_i$ the gap between $I_{i}$ and $I_{i-1}$, where $i>2$. 
Let $Y_j$ be random vectors  measurable with respect to $\sigma\{X_k: k\in I_j\}$. Let $\{\mathcal C_1,...,\mathcal C_d\}$ be a partition of $\{1,...,\ell\}$. Define $Y_{\cC}=(Y_j)_{j\in\mathcal C}$ and let $\widetilde Y_{\cC_j}$ be independent copies of $Y_{\cC_j}$. Let $H=H(y_1,...,y_\ell)$ be a measurable function and suppose that $\mathbb E[|H(Z_1,...,Z_\ell)|]<\infty$ where $Z_j$ is either $Y_j$ or $\widetilde Y_j$, where $\widetilde Y_j, 1\leq j\leq \ell$ are independent copies of $Y_j$. Let $\gamma$ be the maximum of these $2^\ell$ expectations.
Let us view $H$ as a function 
$H=H(y_{\cC_1},...,y_{\cC_d})$ where $y_{\cC}=(y_j)_{j\in\cC}$. Then
$$
\left|\mathbb E[H(Y_{\cC_1},...,Y_{\cC_d})]-\mathbb E[H(\widetilde Y_{\cC_1},...,\widetilde Y_{\cC_d})]\right|\leq 2\gamma\sum_{k=2}^{d}\psi(d_k).
$$
\end{corollary}
Now the proof of Theorem \ref{Psi} follows by the same arguments 
as those used
 in the proof of Theorem \ref{PhiMix}. Indeed, we first reduce to the case when $\gamma_4\leq 1$ and then we replace any use of Proposition \ref{YeorStien} by Corollary \ref{Corrr}, which only requires the  $L^1$ norms of all the terms to be bounded. In the latter we used that when $\gamma_3\leq 1$ then all the terms appearing in the expectations in the terms $\zeta$ and $\cE_i,\, 1\leq i\leq 6$ have uniformly bounded $L^1$ norms.  
\qed

\begin{remark}
Using only the boundedness of $\delta_f$, under the third moment condition $\sup_j\|X_j\|_{L^3}<\infty$ our proof yield the suboptimal rate $O(NV_N^{-3/2}+\sqrt{N}V_N^{-1})$ which is $O(N^{-1/4})$ in the case $V_N$ grows linearly fast in $N$.     
\end{remark}

\subsection{Application to non-stationary uniformly bounded $\psi$-mixing sequences}
In this section we prove the following result.
\begin{theorem}\label{Psi1}
Suppose that $\sup_j\|X_j\|_{L^\infty}<\infty$ and that $V_N\to\infty$ as $N\to\infty$. Assume also that 
$$
\sum_{n=1}^\infty n^2\psi(n)<\infty.
$$
Then with $W_N=(V_N)^{-1/2}S_N$,
$$
\cW_2(W_N, N(0,1))=O(V_N^{-1/2}).
$$
\end{theorem}

\begin{proof}
In \cite[Section 3]{Haf}, we showed that for all $A>1$ large enough we can construct a sequence of intervals $I_1,I_2,...$ in the positive integers such that $I_i$ is to the left of $I_{i+1}$ and the union of the $I_i$ is $\mathbb N$ such that for all $i$,
$$
A\leq \sup_{I\subset I_i}\left\|\sum_{n\in I}X_n\right\|_{L^2}\leq 2A
$$
where $I$ is a subinterval if $I_i$ with the same left endpoint.
Moreover, let $m_N$ be the number of intervals $I_i$ such that $I_i$ is contained in $\{1,...,N\}$. Then we showed that there are constants $C_1,C_2>0$ which depend only on some upper bound on $\sup_j\|X_j\|_\infty$ such that for all $N$,
$$
C_1V_N\leq m_N\leq C_2V_N.
$$
Furthermore, we showed that for all $2\leq p<\infty$ there is a constant $c_p=c_p(A)$ such that for all $i$ and $I$ as above we have
$$
\left\|\sum_{n\in I}X_n\right\|_{L^p}\leq c_p.
$$

Now, fix $N\in\mathbb N$ and  define $B_{i,N}=I_i$ if $i<m_N$ and $B_{m_N,N}=(B_{m_N}\cup B_{m_N+1})\cap\{1,2,...,N\}$. This induces a decomposition of $\{1,...,N\}$ into intervals $B_{i,N}, 1\leq i\leq m_N$ in the positive integers such
and such that for all $2\leq p<\infty$ and all $i\leq m_N$,
$$
\left\|\sum_{n\in B_{i,N}}X_n\right\|_{L^p}\leq 2c_p.
$$
Now the theorem follows by applying Theorem \ref{Psi} with the random variables 
$$
\widetilde X_{i}=\sum_{n\in B_{i,N}}X_n,\,\, i\leq m_N
$$ 
and $\widetilde X_i=0$ for $i>m_N$, and using that the estimates in Theorem \ref{Psi} hold with an absolute constant. 
\end{proof}
Note that Theorem \ref{Psi1} seems to be the first time optimal CLT rates in any form are derived for non-stationary sequences which are not Markov dependent or formed by an exponentially mixing chaotic nonstationary dynamical system.

\subsection*{Acknowledgments}
JL was supported by JSPS via the project LEADER.

\end{document}